\documentclass[a4paper,reqno]{amsart}

\usepackage[foot]{amsaddr}
\usepackage{graphicx,enumerate,nicefrac,color,bm,cite}
\usepackage{amsmath, amssymb, amstext, amsfonts}
\usepackage{stmaryrd}
\usepackage[dvipsnames]{xcolor}
\usepackage{comment}
\usepackage{geometry}
\usepackage{subfig}
\usepackage{a4wide}
\usepackage{mathrsfs}
\usepackage{algorithm,algpseudocode,tabularx}

\makeatletter
\newcommand{\multiline}[1]{%
  \begin{tabularx}{\dimexpr\linewidth-\ALG@thistlm}[t]{@{}X@{}}
    #1
  \end{tabularx}
}
\makeatother

\newcommand{\Lom}[1]{\mathrm{L}^{#1}(\Omega)}
\newcommand{\norm}[1]{\left\|#1\right\|}
\newcommand{\abs}[1]{\left|#1\right|}
\newcommand{\operator}[1]{\mathsf{#1}}
\newcommand{\B}{\operator{B}}

\newcommand{\E}{\operator{E}}
\newcommand{\F}{\mathfrak{F}}

\renewcommand{\L}{\mathrm{L}} 
\renewcommand{\H}{\mathrm{H}}
\newcommand{\W}{\mathrm{W}}
\newcommand{\VV}{\spc{V}}
\newcommand{\WW}{\spc{W}}
\newcommand{\R}{\operator{R}}
\newcommand{\T}{\operator{T}_{\Delta t}}
\newcommand{\G}{\Lambda_f}
\renewcommand{\d}{\operator{d}}
\newcommand{\ds}{\,\d s}

\newcommand{\dx}{\,\d\x}

\newcommand{\NN}[1]{\left|\!\left|\!\left|#1\right|\!\right|\!\right|}
\newcommand{\dprod}[2]{\left<#1,#2\right>}

\newcommand{\x}{\bm{\mathsf{x}}}
\renewcommand{\l}{\ell}
\newcommand{\spc}[1]{\mathbb{#1}}
\newcommand{\CC}{\rho}

\renewcommand{\u}[1]{\bm{\mathfrak{u}}_{N_{#1}}}
\newcommand{\up}[1]{\bm{\mathfrak{u}}_{N_{#1}+1}}
\newcommand{\um}{\bm{\mathfrak{u}}_{N-1}}
\newcommand{\dof}{\mathtt{dof}}
\newcommand{\errH}{\mathtt{err}_{N}}

\newcommand{\CP}{C_{\mathrm{P}}}
\renewcommand{\P}{\omega}
\newcommand{\Pref}{\widetilde{\P}}
\newcommand{\Vh}{\widehat{\mathbb{V}}}

\DeclareMathOperator{\Span}{span}
\DeclareMathOperator*{\esssup}{ess\,sup}
\DeclareMathOperator{\Ei}{Ei}

\newtheorem{theorem}{Theorem}[section]
\newtheorem{lemma}[theorem]{Lemma}
\newtheorem{proposition}[theorem]{Proposition} 
\newtheorem{cor}[theorem]{Corollary}

\newtheorem{assumption}[theorem]{Assumption}
\theoremstyle{definition}
\newtheorem{experiment}{Experiment}[section]

\newtheorem{remark}[theorem]{Remark}

\title[Adaptive energy minimisation for semilinear PDE]{A numerical energy minimisation approach for semilinear diffusion-reaction boundary value problems based on steady state iterations}

\author[M.~Amrein]{Mario Amrein}\address{Institute for Risk and Insurance, Applied University of Zurich, Technoparkstrasse 2, CH-8400 Winterthur}
\author[P.~Heid]{Pascal Heid}\address{Mathematical Institute, University of Oxford, Woodstock Road, Oxford OX2
6GG, UK}
\author[T.~P.~Wihler]{Thomas P.~Wihler}\address{Mathematics Institute, University of Bern, Sidlerstr. 5, CH-3012 Bern,
Switzerland}

\email{mario.amrein@zhaw.ch \and pascal.heid@maths.ox.ac.uk \and wihler@math.unibe.ch}

\thanks{The authors acknowledge the financial support of the Swiss National Science Foundation (SNF),
Grant No. 200021\underline{\phantom{x}}182524, and Project No. P2BEP2\underline{\phantom{x}}191760}

\begin{document}

\begin{abstract}
We present a novel energy-based numerical analysis of semilinear diffusion-reaction boundary value problems. Based on a suitable variational setting, the proposed computational scheme can be seen as an energy minimisation approach. More specifically, this procedure aims to generate a sequence of numerical approximations, which results from the iterative solution of related (stabilised) linearised discrete problems, and tends to a local minimum of the underlying energy functional. Simultaneously, the finite-dimensional approximation spaces are adaptively refined; this is implemented in terms of a new mesh refinement strategy in the context of finite element discretisations, which again relies on the energy structure of the problem under consideration, and does not involve any a posteriori error indicators. In combination, the resulting adaptive algorithm consists of an iterative linearisation procedure on a sequence of hierarchically refined discrete spaces, which we prove to converge towards a solution of the continuous problem in an appropriate sense. Numerical experiments demonstrate the robustness and reliability of our approach for a series of examples.
\end{abstract}



\maketitle

\section{Introduction}
We develop and analyse a new iterative linearised finite element discretisation approach for semilinear elliptic diffusion-reaction equations. Specifically, on an open and bounded polytopal domain $\Omega\subset\mathbb{R}^d$, $d\in\{1,2,3\}$, with boundary $\partial \Omega$ consisting of straight faces, and for a (possibly nonlinear) reaction term $f:\,\Omega\times\mathbb{R}\to\mathbb{R}$, we aim to numerically approximate solutions~$u:\,\Omega \to\mathbb{R}$ of the boundary value model problem
\begin{equation}\label{eq:poisson}
\begin{aligned}
\Delta u(\x) +f(\x,u(\x))&=0 \quad &&\x\in \Omega,\\
u(\x)&=0 \quad &&\x\in \partial \Omega. 
\end{aligned}
\end{equation}
We pursue a novel energy-based avenue that exploits the variational structure of~\eqref{eq:poisson} in both analytical and numerical aspects. More precisely, our approach consists of three key parts, which will be outlined briefly in the sequel.

Firstly, in order to provide a suitable theoretical framework, we devise a new energy analysis for the boundary value problem~\eqref{eq:poisson}, which \emph{does not} require any monotonicity or convexity properties on the nonlinear reaction term~$f$. We assume that $f$ features asymptotically linear growth in the second argument (as $|u|\to\infty$), and, thereby, gives rise to a number of relevant applications: We mention, for instance, the sine-Gordon model, where $f(u)\sim -\sin(u)$, which originated from 19th century surface geometry and was rediscovered in various areas of modern physics, see, e.g., \cite{cuevas2014sine}; another example is the Arrhenius type production term, $f(u)\sim (1-|u|)\exp(-\nicefrac{c}{|u|})$, with $c>0$, which appears in chemical diffusion-reaction models (including combustion), see, e.g., \cite{cencini2003reaction,kapila1980reactive}. 

The second building block is a discrete time stepping scheme that is based on viewing solutions of~\eqref{eq:poisson} as steady-state approximations (for $t\to\infty$) of the semilinear parabolic evolution problem
\begin{equation}
\label{eq:02}
\begin{aligned}
\partial_{t} v(\x,t) &=\Delta v(\x,t)+f(\x,v(\x,t)) \quad &&(\x,t)\in\Omega \times (0,\infty),\\
v(\x,t)&=0 \quad &&(\x, t)\in\partial \Omega \times (0,\infty),\\
v(\x,0)&=v_{0}(\x) \quad &&\x\in \Omega,
\end{aligned}
\end{equation}
for a suitable initial guess $v_{0}:\Omega \to \mathbb{R}$ (with zero boundary values). 
%
An unpretentious way to discretise~\eqref{eq:02} with respect to time is the forward Euler scheme (with a time step $\Delta t>0$). It yields an iteratively generated sequence $\{u^{n}\}_{n}$ that is obtained by solving the \emph{linear} elliptic problem
\begin{equation}
\label{eq:03}
\begin{aligned}
\frac{1}{\Delta t} u^{n+1}(\x)-\Delta u^{n+1}(\x) &=\frac{1}{\Delta t}u^n(\x)+f(\x,u^n(\x)) \quad &&(\x,t)\in\Omega \times (0,\infty),\\
u^{n+1}(\x)&=0 \quad &&(\x, t)\in\partial \Omega \times (0,\infty),
\end{aligned}
\end{equation}
for each $n\ge 0$; note that this procedure could also be seen as a (low-order) \emph{stabilised linear} fixed-point iteration for the nonlinear problem~\eqref{eq:poisson}, viz.
\[
\begin{aligned}
\left(-\Delta+\gamma\,\mathsf{id}\right)u^{n+1}&=f(\cdot,u^{n})+\gamma u^n \quad &&\text{in } \Omega,\\
u^{n+1}&=0 \quad &&\text{on } \partial \Omega, 
\end{aligned}
\]
for $n\ge 0$, where $\gamma>0$ take the role of a stability parameter. Under certain conditions, convergence (in a suitable sense) to a solution of the nonlinear equation~\eqref{eq:poisson} can be established. This observation can be exploited for both theoretical as well as for practical purposes. Indeed, we refer, for instance, to the monotone method of sub- and supersolutions in the theory of semilinear partial differential equations, see, e.g.~\cite[\S 9.3]{evans:98}; we also point to related discrete versions in terms of finite differences, cf.~\cite{McKenna:86, Pao:2003}, where the relevant monotonicity properties carry over to the finite-dimensional framework. In either of these approaches, the availability of a suitable (continuous resp.~discrete) maximum principle is crucial. As a consequence, whenever numerical approximation methods which are not defined in a point-wise manner (such as, e.g., finite element or spectral discretisations) are employed, then the monotonicity approach cannot be applied in an obvious way. In the present paper, we circumvent this issue by making use of the underlying variational framework associated to~\eqref{eq:poisson}, and prove that the sequence $\{u^n\}_n$ resulting from~\eqref{eq:03} features some favourable energy properties; these, in turn, allow to establish an alternative convergence analysis. We remark in passing that, instead of using the forward Euler method (or another explicit time marching scheme) for the approximation of steady-state solutions to~\eqref{eq:02}, the backward Euler method could also be of interest in light of its unconditional stability. Evidently, this approach requires the application of a suitable nonlinear solver in each discrete temporal step. For instance, combining the backward Euler discretisation with the Newton iteration scheme, the so-called pseudo-transient-continuation (PTC) method presented in \cite{Deuflhard:04} emerges; we also refer to \cite{AmreinWihler:17PTC}, where the PTC approach was investigated in the specific context of semilinear singularly perturbed problems.

The third component of the proposed numerical procedure in this work concerns the application of an efficient adaptive finite element mesh refinement strategy. We emphasise that this aspect is of particular importance in the context of semilinear equations~\eqref{eq:poisson} as solutions may exhibit local singular effects including boundary layers, interior shocks, or (multiple) spikes. Again, we will resort to the variational structure in order to use a new energy-driven adaptive finite element mesh refinement technique that has been proposed recently in the context of the (semilinear) Gross-Pitaevskii eigenvalue equation~\cite{HeidStammWihler:19}. In contrast to traditional approaches, we point out that this methodology \emph{does not} require any a posteriori error indicators to drive the adaptive process. 

The steady-state iteration~\eqref{eq:03} is performed on a sequence of hierarchically enriched finite element spaces, which, in turn, are obtained from an energy-based adaptive mesh refinement procedure as mentioned above. In order to realise these ideas within an efficient computational algorithm, they will be effectively combined in terms of a simultaneous interplay. Roughly speaking, we reinitiate the iteration~\eqref{eq:03} on a locally enriched discrete space as soon as the potential energy change in the approximate solution becomes comparable to the iteration error on the current space. Under a natural assumption on the adaptive meshes in each refinement step, we prove that the proposed algorithm converges (in an appropriate sense) to a weak solution of~\eqref{eq:poisson}. We remark that our approach, i.e. the intertwined application of linear iterations and adaptive discretisation methods, is closely related to the recent developments on the (adaptive) iterative linearised Galerkin (ILG) methodology~\cite{HeidWihler:22,HeidWihler:20,HeidWihler:19v2,HeidWihler2:19v1,HeidPraetoriusWihler:2021,CongreveWihler:17,AmreinWihler:15,HoustonWihler:18}; we also refer to the seminal works~\cite{ErnVohralik:13,El-AlaouiErnVohralik:11,BernardiDakroubMansourSayah:15,GarauMorinZuppa:11}. 

\subsection*{Outline}
In \S\ref{sec:probform} we introduce the weak formulation of~\eqref{eq:poisson}, and develop some crucial energy properties. Moreover, in \S\ref{sec:energy_minimisation} the linearised iteration~\eqref{eq:03} is analysed in a Hilbert (sub-)space setting, and some energy-related convergence results will be established.
In \S\ref{sec:adaptiveIP}, we present the main algorithm in general form, which involves an effective interplay of the iterative linearisation scheme~\eqref{eq:03} and adaptive discretisations thereof in terms of arbitrary Galerkin spaces; in addition, we provide a new convergence analysis. Subsequently, we discuss the energy-driven adaptive mesh refinement procedure in the specific context of the finite element method, and perform some numerical experiments. Finally, we summarize our work in \S\ref{sec:concl}.

\section{Variational framework} \label{sec:probform}

\subsection{Function spaces and norms} For the purpose of this paper, we define the space~$\VV:=\H_{0}^{1}(\Omega)$, the standard Sobolev subspace of functions in $\H^1(\Omega)=\W^{1,2}(\Omega)$ with zero trace on~$\partial\Omega$. The space~$\VV$ is equipped with a parametrised class of norms~$\NN{\cdot}_{\lambda}$, where, for $\lambda > 0$, we define
\begin{align} \label{eq:norm}
\NN{v}_{\lambda}:=\Bigl(\lambda\norm{\nabla v}_{\Lom{2}}^2 +\norm{v}_{\Lom{2}}^2 \Bigr)^{\nicefrac{1}{2}},\qquad v\in \VV.
\end{align}
Here, $\|\cdot\|_{\Lom{2}}$ denotes the $\L^2$-norm on~$\Omega$. Observing the Poincar\'e inequality,
\begin{equation}\label{eq:PI0}
\norm{v}^2_{\Lom{2}}\le\CP\norm{\nabla v}^2_{\Lom{2}}\qquad\forall v\in\VV,
\end{equation}
with a constant $\CP>0$ only depending on $\Omega$, we devise the following result for the norm $\NN{\cdot}_{\lambda}$.

\begin{lemma}[Poincar\'e inequality]\label{lem:P}
For any $v\in\VV$ and $\lambda >0$ it holds the bound
\begin{subequations}\label{eq:PI}
\begin{align}
\norm{v}_{\Lom{2}}&\le\beta(\lambda)^{\nicefrac12}\NN{v}_\lambda,
\intertext{with}
\beta(\lambda):&=\frac{\CP}{\CP+\lambda},\label{eq:beta}
\end{align}
\end{subequations}
where $\CP$ is the Poincar\'e constant from \eqref{eq:PI0}.
\end{lemma}

\begin{proof}
For $v\in\VV$ and $\lambda > 0$, using \eqref{eq:PI0}, we have 
\begin{align*}
\norm{v}^2_{\Lom{2}}
&=\frac{1}{\CP+\lambda}\left(\CP\norm{v}^2_{\Lom{2}}+\lambda\norm{v}^2_{\Lom{2}}\right)
\le\frac{1}{\CP+\lambda}\left(\CP\norm{v}^2_{\Lom{2}}+\CP\lambda\norm{\nabla v}^2_{\Lom{2}}\right),
\end{align*}
which yields the claim.
\end{proof}

\subsection{Energy functional}\label{sc:EF}

We define an (energy) functional $\E:\,\VV \to \mathbb{R}$ associated with \eqref{eq:poisson} by 
\begin{align} \label{eq:energyfunctional}
\E(u)=\frac{1}{2} \int_\Omega |\nabla u|^2 \dx - \int_\Omega \F(\x,u(\x)) \dx,
\end{align} 
where, for $\x\in\Omega$, we let
\begin{align*}
\F(\x,t)=\int_0^t f(\x,s) \, \d s,\qquad t\in\mathbb{R}.
\end{align*}
Let $\VV^\prime$ denote the dual space of $ \VV$. Then, for any $u \in \VV$, a straightforward calculation reveals that the G\^ateaux derivative of $\E$ is given by
\begin{align} \label{eq:energyderivative}
\dprod{\E'(u)}{v}:=\int_\Omega \left( \nabla u \cdot \nabla v - f(\x,u)v\right)\dx \qquad \forall v \in \VV,
\end{align}
where $\dprod{\cdot}{\cdot}$ denotes the duality pairing in $\VV^\prime\times\VV$. Hence, the Euler--Lagrange equation of the minimisation problem
\begin{align} \label{eq:minproblem}
u\in\VV:\qquad \E(u)=\min_{v \in \VV} \E(v),
\end{align}
is given in weak form by
\begin{align} \label{eq:poissonweak}
u\in \VV:\qquad \dprod{\E'(u)}{v}=\int_\Omega \left(\nabla u \cdot \nabla v \dx - f(\x,u)v\right) \dx=0 \qquad \forall v \in \VV;
\end{align}
in particular, any critical point (especially, any minimiser) of $\E$ in $\VV$ is a solution to~\eqref{eq:poissonweak}, or equivalently, a weak solution of~\eqref{eq:poisson}.

We introduce the following structural assumptions on the nonlinearity $f$ present in \eqref{eq:poisson}, which are crucial for the remainder of this work.

\begin{assumption}[Nonlinearity $f$]\label{apt:A1}
The function $f:\Omega \times \mathbb{R} \to \mathbb{R} $ satisfies the following properties:
\begin{enumerate}[\rm(i)]
\item $f(\cdot,0) \in \Lom{2}$.
\item $f$ is differentiable in the second variable. 
\item There exists a constant $\CC>0$ such that the set
\begin{equation}\label{eq:set}
\G(\CC):=\left\{\lambda>0:\,\sigma_f(\lambda)< \CC+\nicefrac{1}{\lambda}\right\}
\end{equation}
is non-empty, where we let
\begin{equation}\label{eq:sigma}
\sigma_f(\lambda):=\esssup_{\x \in \Omega} \sup_{u \in \mathbb{R}} \left|\frac{\partial f}{\partial u}(\x,u)+\frac{1}{\lambda}\right|,\qquad\lambda>0.
\end{equation}
\end{enumerate}
\end{assumption}


\begin{remark}
If Assumption \ref{apt:A1} is fulfilled, then the function $g_\lambda$, for $\lambda\in\G(\CC)$, defined by
\begin{equation}\label{eq:g}
g_\lambda(\x,u):=f(\x,u)+\lambda^{-1}u,\qquad (\x,u)\in\Omega\times\mathbb{R},
\end{equation} 
satisfies the uniform Lipschitz continuity bound
\begin{equation}\label{eq:dg}
|g_\lambda(\x,u)-g_\lambda(\x,v)|\le\sigma_f(\lambda)|u-v|\qquad\forall u,v\in\mathbb{R},
\end{equation}
for almost every $\x \in \Omega$, where $\sigma_f(\lambda)<\CC+\nicefrac{1}{\lambda}$, cf.~\eqref{eq:set}.
\end{remark}

\begin{lemma}[Energy representation]\label{lem:Ebound}
Under the Assumption~\ref{apt:A1}, for $\rho>0$ and $\lambda\in\G(\CC)$, the energy functional from~\eqref{eq:energyfunctional} can be represented by 
\begin{equation}\label{eq:repr}
\E(v)=\frac{1}{2\lambda}\NN{v}^2_{\lambda}+\Xi(v)\norm{v}_{\Lom{2}}\qquad\forall v\in\VV,
\end{equation}
where $\Xi:\,\VV\to\mathbb{R}$ is a function that is bounded by
\[
|\Xi(v)|\le\frac12\sigma_f(\lambda)\norm{v}_{\Lom{2}}+\norm{f(\cdot,0)}_{\Lom{2}}\qquad\forall v\in\VV;
\]
in particular, the energy functional $\E$ from~\eqref{eq:energyfunctional} is well-defined on $\VV$.
\end{lemma}

\begin{proof}
For any $t\in\mathbb{R}$, we note that
\begin{align*}
\F(\x,t)
&=\int_0^t \left(g_\lambda(\x,s)-g_\lambda(\x,0)\right)\ds+f(\x,0)t-\frac{t^2}{2\lambda}.
\end{align*}
Applying~\eqref{eq:dg}, for almost every $\x\in\Omega$, we observe the bound
\begin{equation}\label{eq:20211125a}
\left|\F(\x,t)+\frac{t^2}{2\lambda}\right|
\le\int_0^t\left|g_\lambda(\x,s)-g_\lambda(\x,0)\right|\ds+|f(\x,0)t|
\le\frac{1}{2}\sigma_f(\lambda)t^2+|f(\x,0)t|.
\end{equation}
Moreover, for any $v\in\VV$, $v\neq 0$, letting
\[
\Xi(v):=-\norm{v}^{-1}_{\Lom{2}}\int_\Omega\left(\F(\x,v(\x))+\frac{1}{2\lambda}v(\x)^2\right)\dx,
\]
we immediately derive the representation~\eqref{eq:repr}. Here, employing~\eqref{eq:20211125a} and using the Cauchy-Schwarz inequality, we infer that
\begin{align*}
\norm{v}_{\Lom{2}}|\Xi(v)|
&\le
\int_\Omega\left|\F(\x,v(\x))+\frac{1}{2\lambda}v(\x)^2\right|\dx
\le\frac12\sigma_f(\lambda)\norm{v}^2_{\Lom{2}}+\norm{f(\cdot,0)}_{\Lom{2}}\norm{v}_{\Lom{2}}.
\end{align*}
This yields the result.
\end{proof}

\begin{remark}\label{rem:wc}
From Lemma~\ref{lem:Ebound} and upon applying Lemma~\ref{lem:P}, for any $v\in\VV$, we infer the lower bound
\begin{align*}
\E(v)&\ge\frac{1}{2}\left(\frac{1}{\lambda}-\beta(\lambda)\sigma_f(\lambda)\right)\NN{v}^2_\lambda-\beta(\lambda)^{\nicefrac12}\norm{f(\cdot,0)}_{\Lom{2}}\NN{v}_{\lambda}.
\end{align*}
Invoking the definition of the set $\G(\CC)$ from~\eqref{eq:set}, we notice that
\[
\frac{1}{\lambda}-\beta(\lambda)\sigma_f(\lambda)>\frac{1-\CP\CC}{\CP+\lambda}.
\]
Hence,  if $\rho\le\nicefrac{1}{\CP}$ then it follows that 
\begin{equation}\label{eq:wc}
\E(v)\to+\infty\qquad\text{whenever}\qquad \NN{v}_{\lambda}\to\infty;
\end{equation}
this property is referred to as the \emph{weak coercivity} of $\E$.
\end{remark}

The following result is instrumental for the analysis below.

\begin{lemma}[Energy expansion]\label{lem:EE}
Suppose that Assumption \ref{apt:A1} is satisfied for some $\CC>0$, and consider $\lambda\in\G(\CC)$, cf. \eqref{eq:set}. Then, for any $u,v\in\VV$, 
it holds that
\[
\E(v)-\E(u)
=\dprod{\E'(u)}{v-u}+\frac{1}{2\lambda}\NN{v-u}^2_\lambda+\Psi_\lambda(u,v),
\]
where $\E'$ is the derivative from \eqref{eq:energyderivative}, and $\Psi_\lambda(u,v)$ is a remainder term that satisfies the bound
\begin{align*} 
|\Psi_\lambda(u,v)|\le\frac12\sigma_f(\lambda)\norm{v-u}^2_{\Lom{2}},
\end{align*}
with $\sigma_f(\lambda)$ from \eqref{eq:sigma}.
\end{lemma}

\begin{proof}
Given $u,v\in\VV$, we define $\delta:=v-u$. Then, by the main theorem of calculus, we have
\[
\E(v)-\E(u)
=\int_0^1\frac{\d}{\ds}\E(u+s\delta)\ds
=\int_0^1\dprod{\E'(u+s\delta)}{\delta}\ds.
\]
Recalling \eqref{eq:energyderivative}, we infer that
\begin{align*}
\E(v)-\E(u)
&=\int_0^1\int_\Omega \left( \nabla (u+s\delta) \cdot \nabla \delta - f(\x,u+s\delta)\delta\right)\dx\ds\nonumber\\
&=\dprod{\E'(u)}{\delta}
+\frac12\norm{\nabla\delta}^2_{\Lom{2}}
-\int_0^1\int_\Omega\left(f(\x,u+s\delta)-f(\x,u)\right)\delta\dx\ds.
\end{align*}
Hence, by definition of the function $g_\lambda$ from \eqref{eq:g}, we can write
\[
\E(v)-\E(u)
=\dprod{\E'(u)}{\delta}
+\frac12\norm{\nabla\delta}^2_{\Lom{2}}+\frac{1}{2\lambda}\norm{\delta}^2_{\Lom{2}}
-\int_0^1\int_\Omega\left(g_\lambda(\x,u+s\delta)-g_\lambda(\x,u)\right)\delta\dx\ds.
\]
Invoking \eqref{eq:dg} we observe that
\[
\left|\int_0^1\int_\Omega\left(g_\lambda(\x,u+s\delta)-g_\lambda(\x,u)\right)\delta\dx\ds\right|
\le \sigma_f(\lambda)\int_0^1\int_\Omega s|\delta|^2\dx\ds
=\frac12\sigma_f(\lambda)\norm{\delta}^2_{\Lom{2}},
\]
which shows the claim.
\end{proof}
 
We conclude this section with the ensuing observation.

\begin{lemma}[Lipschitz continuity of $\E'$] \label{lem:aux2}
Given Assumption~\ref{apt:A1} for some $\rho>0$, and $\lambda\in\G(\CC)$, cf. \eqref{eq:set}. Then $\E':\VV \to \VV'$ is (uniformly) Lipschitz continuous in the sense that
\begin{align} \label{eq:lipschitz}
|\dprod{\E'(u)-\E'(v)}{w}| \leq L_{\E'}(\lambda) \NN{u-v}_{\lambda} \NN{w}_{\lambda} \qquad \forall u,v,w \in \VV,
\end{align}
where 
\begin{equation}\label{eq:L}
L_{\E'}(\lambda)=\frac{1}{\lambda} + \beta(\lambda)\sigma_f(\lambda),
\end{equation}
with $\beta(\lambda)$ from~\eqref{eq:beta}.
\end{lemma}

\begin{proof}
Let $u,v,w\in\VV$. With the aid of the Cauchy-Schwarz inequality we have that
\[
\left|\int_\Omega\nabla(u-v)\cdot\nabla w\dx\right|
\le\norm{\nabla(u-v)}_{\Lom{2}}\norm{\nabla w}_{\Lom{2}}.
\]
Furthermore, for any $\lambda\in\G(\CC)$, using~\eqref{eq:dg}, it follows that
\begin{align} \label{eq:flipschitz}
\begin{split}
\left|\int_\Omega\left(f(\x,u)-f(\x,v)\right)w\dx\right|
&\le\int_\Omega|g_\lambda(\x,u)-g_\lambda(\x,v)||w|\dx
+\frac{1}{\lambda}\int_\Omega\left|u-v\right||w|\dx\\
&\le\left(\sigma_f(\lambda)+\frac{1}{\lambda}\right)\norm{u-v}_{\Lom{2}}\norm{w}_{\Lom{2}}.
\end{split}
\end{align}
Hence, from~\eqref{eq:energyderivative}, we obtain
\begin{align*}
|\dprod{\E'(u)-\E'(v)}{w}|
&\le \norm{\nabla(u-v)}_{\Lom{2}}\norm{\nabla w}_{\Lom{2}}
+\left(\sigma_f(\lambda)+\frac{1}{\lambda}\right)\norm{u-v}_{\Lom{2}}\norm{w}_{\Lom{2}}\\
&\le\frac{1}{\lambda}\NN{u-v}_{\lambda}\NN{w}_{\lambda}+\sigma_f(\lambda)\norm{u-v}_{\Lom{2}}\norm{w}_{\Lom{2}}.
\end{align*}
Applying Lemma~\ref{lem:P} completes the argument.
\end{proof}

\section{Iterative energy minimisation}
\label{sec:energy_minimisation}
In this section, we will present an iterative variational approach for the minimisation problem~\eqref{eq:minproblem}, and establish some energy-related convergence properties.

\subsection{Iteration scheme}\label{sc:it}

We begin by introducing an iterative scheme for the solution of~\eqref{eq:poissonweak}. For this purpose, we pursue the idea of approximating a solution to~\eqref{eq:02} by means of the (discrete) iteration~\eqref{eq:03}, with a fixed time step $\Delta t>0$. In weak form, given $u^{n-1}\in\VV$, for $n\ge 1$, we seek $u^{n} \in \VV$ such that 
\begin{equation}
\label{eq:04}
{\frac{1}{\Delta t}\int_{\Omega}(u^{n}-u^{n-1})v\dx}+\int_{\Omega}{\nabla u^{n}\cdot\nabla v\dx}=\int_{\Omega}{f(\x,u^{n-1})v \dx} \qquad \forall v \in \VV.
\end{equation}
For $n=0$, we let $u^{0}\in\Lom{2}$ be a suitable initial guess. We emphasise that the iteration~\eqref{eq:04}, for given $u^n$, is a \emph{linear} problem for~$u^{n+1}$, and can thus be viewed as a \emph{iterative linearisation} of~\eqref{eq:poisson}. Notice that it can be written equivalently as 
\begin{align} \label{eq:iteration}
\B_{\Delta t}(u^{n},v)=\l_{\Delta t }(u^{n-1};v) \qquad \forall v \in \VV,
\end{align}
where, for $\lambda >0$, we define the \emph{bilinear form} $\B_{\lambda}:\,\VV \times \VV \to \mathbb{R}$ by 
\begin{equation}
\label{eq:bilinear}
\B_{\lambda}(u,v):=\int_{\Omega}{(\lambda  \nabla u \cdot \nabla v+uv)\dx}, \qquad u,v \in \VV,
\end{equation}
and, for given $y \in \Lom{2}$, the \emph{linear form} $\l_\lambda(y;\cdot):\,\VV \to \mathbb{R}$ by
\begin{align} \label{eq:linearform}
\l_{\lambda}(y;v):=\int_{\Omega}{\left(y+\lambda f(\x,y)\right)v\dx},\qquad v\in\VV.
\end{align}
Recalling \eqref{eq:energyderivative}, for all $u\in\VV$, we notice the identity
\begin{equation}\label{eq:Eid}
\dprod{\E'(u)}{v}:=\frac{1}{\lambda}\left(\B_\lambda(u,v)-\l_\lambda(u;v)\right) \qquad \forall v \in \VV.
\end{equation}

\begin{lemma}[Boundedness of $\l_\lambda$]\label{lem:l}
Under Assumption~\ref{apt:A1}, for all $\lambda\in\G(\CC)$, cf.~\eqref{eq:set}, and any $y \in \Lom{2}$, the linear form $\ell_\lambda(y;\cdot)$ from~\eqref{eq:linearform} is bounded in the sense that
\[
|\l_\lambda(y;v)|
\le \lambda\beta(\lambda)^{\nicefrac12}\left(\sigma_f(\lambda)\norm{y}_{\Lom{2}}+\norm{f(\cdot,0)}_{\Lom{2}}\right)\NN{v}_{\lambda}\qquad\forall v\in\VV,
\]
with $\beta(\lambda)$ and $\sigma_f(\lambda)$ from \eqref{eq:beta} and \eqref{eq:sigma}, respectively.
\end{lemma}

\begin{proof}
Given $\lambda\in\G(\CC)$ and $y\in\Lom{2}$. Then, for any $v\in\VV$, we have
\begin{equation}\label{eq:lg}
\l_\lambda(y;v)
=\lambda\int_\Omega g_\lambda(\x,y)v\dx,
\end{equation}
with the function $g_\lambda$ from \eqref{eq:g}. Using that $g_\lambda(\cdot,0)=f(\cdot,0)\in\Lom{2}$, we deduce the identity
\[
\l_\lambda(y;v)
=\lambda\int_\Omega (g_\lambda(\x,y)-g_{\lambda}(\x,0))v\dx+\lambda\int_\Omega f(\x,0)v\dx.
\]
Exploiting \eqref{eq:dg} and applying the Cauchy-Schwarz inequality, it follows that
\begin{align*}
|\l_\lambda(y;v)|
\le \lambda\left(\sigma_f(\lambda)\norm{y}_{\Lom{2}}+\norm{f(\cdot,0)}_{\Lom{2}}\right)\norm{v}_{\Lom{2}}.
\end{align*}
Recalling Lemma \ref{lem:P} completes the proof.
\end{proof}

We note that the bilinear form $\B_\lambda$ from~\eqref{eq:bilinear} is an inner product on $\VV \times \VV$ that induces the norm from~\eqref{eq:norm}; in particular, since this norm, for fixed $\lambda>0$, is equivalent to the standard $\H^1$-norm, the space $\VV$ endowed with the inner product from~\eqref{eq:bilinear} is a Hilbert space. Hence, applying the Riesz representation theorem, we immediately obtain the following result.

\begin{proposition}[Well-posedness]\label{pr:wp}
Let Assumption~\ref{apt:A1} be fulfilled, and fix $\Delta t\in\G(\CC)$, cf.~\eqref{eq:set}. Then, for any initial guess $u^0\in\Lom{2}$, the iteration \eqref{eq:iteration} is well-defined for all $n\ge 1$.
\end{proposition}

\subsection{Convergence analysis in closed subspaces}

Let $\WW \subseteq\VV$ be a closed subspace of $\VV$ (e.g., a finite dimensional Galerkin subspace, or $\VV$ itself). We restrict the weak formulation~\eqref{eq:poissonweak} to $\WW$, viz.
\begin{align} \label{eq:weakW}
u\in\WW:\qquad\int_\Omega \nabla u \cdot \nabla v \dx= \int_\Omega f(\x,u)v \dx \qquad \forall v \in \WW.
\end{align}
In accordance with \eqref{eq:iteration}, for an initial guess $u^0 \in \Lom{2}$, we consider the \emph{iterative linearisation} scheme
\begin{align} \label{eq:iterationW}
u^{n+1} \in \WW: \qquad \B_{\Delta t}(u^{n+1},v)=\ell_{\Delta t}(u^{n};v) \qquad \forall v \in \WW,
\end{align}
which, by arguments similar to those leading to Proposition~\ref{pr:wp}, is well-posed for all $n\ge 0$ if $\Delta t\in\G(\CC)$. Furthermore, as before in \S\ref{sc:EF}, it holds that~\eqref{eq:weakW} is the Euler--Lagrange formulation for critical points of the energy functional $\E$ from \eqref{eq:energyfunctional} on the subspace $\WW$. In particular, the weak equation~\eqref{eq:weakW} can be stated equivalently as
\begin{align} \label{eq:eprimeprob}
u \in \WW: \qquad \dprod{\E'(u)}{v}=0 \qquad\forall v\in\WW,
\end{align}
with $\E'$ from \eqref{eq:energyderivative}, and $\dprod{\cdot}{\cdot}$ the duality pairing on $\WW^\prime\times\WW$. 

\subsubsection{The special case $\CC\le\nicefrac{1}{\CP}$}\label{sc:special}

We first consider the situation where Assumption~\ref{apt:A1}~(iii) is fulfilled with $\CC\le\nicefrac{1}{\CP}$, where $\CP>0$ is the constant from the Poincar\'e inequality~\eqref{eq:PI0}.

\begin{theorem}[Convergence for $\CC\le\nicefrac{1}{\CP}$] \label{thm:convergence}
Let Assumption \ref{apt:A1} be satisfied with $\CC\le\nicefrac{1}{\CP}$, and consider $\Delta t\in\G(\CC)$, cf. \eqref{eq:set}. Then, for any initial guess $u^0\in\Lom{2}$, the sequence generated by \eqref{eq:iterationW} converges to a unique limit $u^\star\in\WW$ with respect to the norm $\NN{\cdot}_{\Delta t}$, and it holds $\dprod{\E'(u^\star)}{w}=0$ for all $w\in\WW$.
\end{theorem}

\begin{proof}
By Lemma \ref{lem:l}, for any $u\in\Lom{2}$, the linear form $\l_{\Delta t}(u;\cdot)$ is well-defined and bounded on~$\WW$. Hence, by the coercivity of the bilinear form $\B_{\Delta t}$ from \eqref{eq:bilinear}, the mapping $\T:\,\Lom{2}\to\WW$ defined via the weak formulation
\[
u\mapsto\T(u):\qquad
\B_{\Delta t}(\T(u),v)=\ell_{\Delta t}(u;v) \qquad \forall v \in \WW,
\]
is well-defined. Moreover, for any $u,v\in\WW$, letting $\delta:=\T(u)-\T(v)$ and invoking \eqref{eq:lg}, we note that
\[
\NN{\delta}^2_{\Delta t}=\B_{\Delta t}(\delta,\delta)
=\l_{\Delta t}(u;\delta)-\l_{\Delta t}(v;\delta)
=\Delta t\int_\Omega (g_{\Delta t}(\x,u)-g_{\Delta t}(\x,v))\delta\dx.
\]
Employing \eqref{eq:dg} and using the Cauchy-Schwarz inequality, it follows that
\[
\NN{\delta}^2_{\Delta t}\le \sigma_f(\Delta t)\Delta t\int_\Omega |u-v||\delta|\dx
\le\sigma_f(\Delta t)\Delta t\norm{u-v}_{\Lom{2}}\norm{\delta}_{\Lom{2}}.
\]
Involving Lemma \ref{lem:P}, we infer the stability bound
\[
\NN{\T(u)-\T(v)}_{\Delta t}
\le\beta(\Delta t)\sigma_f(\Delta t)\Delta t\NN{u-v}_{\Delta t}.
\]
Exploiting that 
$
\sigma_f(\Delta t)<\CC+\nicefrac{1}{\Delta t}\le\nicefrac{1}{\CP}+\nicefrac{1}{\Delta t},
$ 
we notice that $\beta(\Delta t)\sigma_f(\Delta t)\Delta t<1$. Hence, we conclude that $\T|_{\WW}:\,\WW\to\WW$ is a contractive operator. Therefore, by Banach's fixed point theorem, the iteration defined by
$u^{n+1}=\T(u^n)$, for $n\ge 0$,
which is exactly~\eqref{eq:iterationW}, converges to a unique fixed point $u^\star\in\WW$. Moreover, for any $w\in\WW$, 
it follows from~\eqref{eq:Eid} that
\[
\dprod{\E'(u^\star)}{w}
=\frac{1}{\Delta t}\left(\B_{\Delta t}(u^\star,w)-\l_{\Delta t}(u^\star;w)\right)
=\frac{1}{\Delta t}\left(\B_{\Delta t}(\T(u^\star),w)-\l_{\Delta t}(u^\star;w)\right)=0,
\]
which completes the argument.
\end{proof}

\subsubsection{The general case}

If the constant $\CC$ in Assumption~\ref{apt:A1}~(iii) is not sufficiently small, then a (weaker) convergence result for the iteration~\eqref{eq:iterationW} can still be established under additional prerequisites, for instance if the energy functional $\E$ from~\eqref{eq:energyfunctional} is weakly coercive (see Remark~\ref{rem:37} below for further details on this matter) and if the nonlinearity $f$ is continuous.

We begin with the following stability result.

\begin{proposition}[Stability]\label{pr:stab}
Let Assumption \ref{apt:A1} be fulfilled for some $\CC>0$, and define 
\[
\mu_f:=\begin{cases}
\left(\sup\G(\CC)\right)^{-1}&\text{if }\sup\G(\CC)<\infty\\
0&\text{otherwise},
\end{cases}
\] 
and $\kappa_f:=\max\left\{\mu_f+\nicefrac{\CC}{2}-\nicefrac{1}{2\CP},0\right\}$.
Then, for any $\Delta t>0$ with $\nicefrac{1}{\Delta t}>\kappa_f$, the sequence $\{u^n\}_{n}\subset\WW$ generated by the iteration \eqref{eq:iterationW} satisfies the estimate
\begin{equation}\label{eq:stab}
\E(u^{n})-\E(u^{n+1})
\ge\gamma_f(\Delta t)\NN{u^{n+1}-u^n}^2_{\Delta t},
\end{equation}
where
\begin{equation*}
\gamma_f(\Delta t):=\min\left\{\frac{\nicefrac{1}{\Delta t}-\kappa_f}{\nicefrac{\Delta t}{\CP}+1},\frac{1}{2\Delta t}\right\}>0;
\end{equation*}
in particular, the sequence $\{\E(u^n)\}_{n}$ is monotone decreasing.
\end{proposition}

\begin{proof}
Let $\delta^n:=u^{n+1}-u^n$, for $n\ge 0$, and consider $\Delta t>0$ with $\nicefrac{1}{\Delta t}>\kappa_f$. Choose $\lambda\in\G(\CC)$ sufficiently large such that 
\begin{equation}\label{eq:20210903a}
\nicefrac{1}{\Delta t}>\max\left\{\nicefrac{1}{\lambda}+\nicefrac{\CC}{2}-\nicefrac{1}{2\CP},0\right\}\ge\kappa_f.
\end{equation}
Exploiting the energy expansion from Lemma \ref{lem:EE}, we have
\[
\E(u^{n+1})-\E(u^{n})
\le\dprod{\E'(u^n)}{\delta^n}+\frac{1}{2\lambda}\NN{\delta^n}^2_{\lambda}+\frac12\sigma_f(\lambda)\norm{\delta^n}^2_{\Lom{2}}.
\]
In addition, invoking \eqref{eq:Eid} and \eqref{eq:iterationW}, for any $w\in\WW$, we notice that 
\begin{equation}\label{eq:20210831a}
\dprod{\E'(u^n)}{w}
=\frac{1}{\Delta t}\left(\B_{\Delta t}(u^n,w)-\l_{\Delta t}(u^n;w)\right)
=\frac{1}{\Delta t}\B_{\Delta t}(u^n-u^{n+1},w)
=-\frac{1}{\Delta t}\B_{\Delta t}(\delta^n,w),
\end{equation}
and therefore
\[
\dprod{\E'(u^n)}{\delta^n}=-\frac{1}{\Delta t}\NN{\delta^n}^2_{\Delta t}.
\]
Combining the above, and applying the bound from~\eqref{eq:set}, yields
\begin{align*}
\E(u^{n})-\E(u^{n+1})
&\ge\frac{1}{\Delta t}\NN{\delta^n}^2_{\Delta t}-\frac{1}{2\lambda}\NN{\delta^n}^2_{\lambda}-\frac12\sigma_f(\lambda)\norm{\delta^n}^2_{\Lom{2}}\\
&=\frac{1}{2}\norm{\nabla\delta^n}^2_{\Lom{2}}
+\frac12\left(\frac{2}{\Delta t}-\frac{1}{\lambda}-\sigma_f(\lambda)\right)\norm{\delta^n}^2_{\Lom{2}}\\
&\ge\frac{1}{2}\norm{\nabla\delta^n}^2_{\Lom{2}}
+\frac12\left(\frac{2}{\Delta t}-\frac{2}{\lambda}-\CC\right)\norm{\delta^n}^2_{\Lom{2}}\\
&\ge\frac{1}{2}\norm{\nabla\delta^n}^2_{\Lom{2}}
+\left(\frac{1}{\Delta t}-\max\left\{\frac{1}{\lambda}+\frac{\CC}{2}-\frac{1}{2\CP},0\right\}-\frac{1}{2\CP}\right)\norm{\delta^n}^2_{\Lom{2}}.
\end{align*}
Defining the positive constants
\[
\eta_\lambda:=\nicefrac{1}{\Delta t}-\max\left\{\nicefrac{1}{\lambda}+\nicefrac{\CC}{2}-\nicefrac{1}{2\CP},0\right\}>0,\qquad
\epsilon_\lambda:=\min\left\{\frac{\eta_\lambda}{\nicefrac{\Delta t}{\CP}+1},\frac{1}{2\Delta t}\right\}>0,
\]
cf. \eqref{eq:20210903a}, it follows that
\begin{align*}
\E(u^{n})-\E(u^{n+1})
&\ge\epsilon_{\lambda}\Delta t\norm{\nabla\delta^n}^2_{\Lom{2}}
+\frac12\left(1-2\epsilon_{\lambda}\Delta t\right)\norm{\nabla\delta^n}^2_{\Lom{2}}
+\frac12\left(2\eta_\lambda-\nicefrac{1}{\CP}\right)\norm{\delta^n}^2_{\Lom{2}}.
\end{align*}
Then, using that $1-2\epsilon_{\lambda}\Delta t\ge 0$, and employing \eqref{eq:PI0}, we obtain
\begin{align*}
\E(u^{n})-\E(u^{n+1})
&\ge\epsilon_{\lambda}\Delta t\norm{\nabla\delta^n}^2_{\Lom{2}}
+\left(\eta_\lambda-\nicefrac{\epsilon_{\lambda}\Delta t}{\CP}\right)\norm{\delta^n}^2_{\Lom{2}}.
\end{align*}
Noticing that
\begin{align*}
\eta_\lambda-\frac{\epsilon_{\lambda}\Delta t}{\CP}
\ge\eta_\lambda-\frac{\Delta t}{\CP}\frac{\eta_\lambda}{\nicefrac{\Delta t}{\CP}+1}
=\frac{\eta_\lambda}{\nicefrac{\Delta t}{\CP}+1}
\ge\epsilon_\lambda,
\end{align*}
we infer the bound
$
\E(u^{n})-\E(u^{n+1})
\ge\epsilon_\lambda\NN{\delta^n}^2_{\Delta t}.
$
Finally, observing that
\[
\lim_{\lambda\nearrow\sup\G(\CC)}\epsilon_\lambda=\gamma_f(\Delta t),
\]
completes the proof.
\end{proof}

\begin{remark}\label{rem:decreasingf}
If there exists a constant $C_f>0$ such that, for almost every $\x \in \Omega$, it holds the bound
\begin{equation}\label{eq:Cf}
\sup_{u \in \mathbb{R}}\left|\frac{\partial f}{\partial u}(\x,u)\right|<C_f,
\end{equation}
then we may choose $\rho=C_f$ in Assumption~\ref{apt:A1}, and we obtain $\sup\G(C_f)=\infty$. In the context of Proposition~\ref{pr:stab} this yields $\kappa_f=\nicefrac12\max\left\{C_f-\nicefrac{1}{\CP},0\right\}$.
Conversely, if Assumption~\ref{apt:A1} is satisfied for some $\rho>0$, then, for any $\lambda\in\G(\rho)$, we have
\begin{equation}\label{eq:20211126a}
\sup_{u \in \mathbb{R}}\left|\frac{\partial f}{\partial u}(\x,u)\right|
\le\frac{1}{\lambda}+\sup_{u \in \mathbb{R}}\left|\frac{\partial f}{\partial u}(\x,u)+\frac{1}{\lambda}\right|\le\frac{1}{\lambda}+\sigma_f(\lambda)<\rho+\frac{2}{\lambda},
\end{equation}
for almost every $\x\in\Omega$, which is~\eqref{eq:Cf} with $C_f=\rho+\nicefrac{2}{\lambda}$.
\end{remark}



\begin{proposition}[Convergence of residual]\label{pr:res}
Let the assumptions of the previous Proposition~\ref{pr:stab} hold, and suppose that the sequence $\{\E(u^n)\}_{n}$ is bounded from below, where $\{u^n\}_{n}\subset\WW$ is generated by the iteration~\eqref{eq:iterationW}. Then 
\begin{align} \label{eq:dualconvergence}
\sup_{w \in \WW \setminus \{0\}} \frac{\abs{\dprod{\E'(u^n)}{w}}}{\NN{w}_{\Delta t}}\to0\qquad\text{as }n\to\infty,
\end{align}
i.e. $\lim_{n\to\infty}\E'(u^n) =0$ in $\WW'$.
\end{proposition}

\begin{proof}
We recall the dual norm
\begin{equation}\label{eq:res}
\norm{\E'(u)}_{\WW^\prime}:=\sup_{w \in \WW \setminus \{0\}} \frac{\abs{\dprod{\E'(u)}{w}}}{\NN{w}_{\Delta t}},\qquad u\in\WW.
\end{equation}
Applying~\eqref{eq:20210831a}, we observe the bound
\begin{equation*}
|\dprod{\E'(u^n)}{w}|
\le\frac{1}{\Delta t}\NN{\delta^n}_{\Delta t}\NN{w}_{\Delta t}\qquad\forall w\in\WW.
\end{equation*}
Thus, using \eqref{eq:stab}, for $n\ge 0$, it follows that
\begin{align*}
\norm{\E'(u^n)}_{\WW^\prime} 
\leq \frac{1}{\Delta t} \NN{\delta^n}_{\Delta t} 
\leq \frac{1}{\Delta t}\left(\frac{\E(u^{n})-\E(u^{n+1})}{\gamma_f(\Delta t)}\right)^{\nicefrac12}.
\end{align*}
Since the sequence $\{\E(u^n)\}_{n}$ is decreasing (Proposition~\ref{pr:stab}) and, by assumption, bounded from below, we deduce that the right-hand side of the above estimate vanishes for $n \to \infty$, which yields the claim.
\end{proof}

\begin{remark}\label{rem:37}
The boundedness from below of the sequence $\{\E(u^n)\}_{n}$, as required in Proposition~\ref{pr:res}, is guaranteed, for instance, if $\E$ is a weakly coercive energy functional, cf.~\eqref{eq:wc}. Owing to Remark~\ref{rem:wc}, this structural property of $\E$ holds if Assumption~\ref{apt:A1} is fulfilled for some $0<\CC\le\nicefrac{1}{\CP}$, cf.~\S\ref{sc:special}, with the constant $\CP>0$ from~\eqref{eq:PI0}. We emphasise that this condition on $\rho$ is sufficient, however, not necessary for the weak coercivity of $\E$. For example, the nonlinearity $f(u)=\sin(u^2)$ renders the energy functional $\E$ weakly coercive, yet, there is no $\rho>0$ for which Assumption~\ref{apt:A1} can be satisfied in this case. Note also that even if Assumption~\ref{apt:A1} can be established, it is not generally possible to do so with small $\rho$; indeed, consider the function $f(u)=\varepsilon^{-1}\exp(-u^2)$, cf.~Experiment~\ref{exp:Lshape} in \S\ref{sc:ce} below, which requires $\rho$ to be arbitrarily large in \eqref{eq:set} if $\varepsilon\to 0$.
\end{remark}

With regards to the convergence in the norm $\NN{\cdot}_{\Delta t}$, we focus first on the iteration \eqref{eq:iteration}, i.e. for~$\WW=\VV$ in~\eqref{eq:iterationW}.

\begin{theorem}[Convergence---general case] \label{thm:alternative}
Let $\WW=\VV$, and $\Delta t>0$ as in Proposition~\ref{pr:stab}. Moreover, suppose that $f$ is continuous on $\overline{\Omega} \times \mathbb{R}$, and that the energy functional $\E$ is weakly coercive, cf.~\eqref{eq:wc}.
Then, for the sequence $\{u_n\}_{n}\subset\VV$ generated by the iterative scheme~\eqref{eq:iteration}, there is a subsequence which converges (strongly) to a solution $u^\star\in\VV$ of~\eqref{eq:poissonweak}.
\end{theorem}

For the proof of the above result, we require the following compactness lemma from the mountain pass theory, see, e.g.~\cite{R86} for details.

\begin{lemma}\label{lem:mp}
Let $f$ in~\eqref{eq:poisson} be continuous on $\overline\Omega\times\mathbb{R}$, and Assumption~\ref{apt:A1} be satisfied for some $\CC>0$. Moreover, consider a bounded sequence $\{v^n\}_{n}$ in $\VV$ with $\E'(v^n)\to 0$ as $n\to\infty$ in $\VV'$. Then $\{v^n\}_{n}$ has a (strongly) converging subsequence in $\VV$.
\end{lemma}

\begin{proof}
We establish a (uniform) linear growth bound on $f$ with respect to the second argument; then, the proof follows from \cite[Prop.~B.35]{R86}. In fact, for $u\in\mathbb{R}$ and any $\x\in\Omega$, by the continuity of $f$ and its differentiability in the second argument, cf.~Assumption~\ref{apt:A1}~(ii), we have
\[
f(\x,u)=f(\x,0)+\int_0^1\frac{\d}{\ds}f(\x,su)\ds
=f(\x,0)+u\int_0^1\partial_uf(\x,su)\ds.
\]
Then, for $\lambda\in\G(\rho)$, recalling~\eqref{eq:20211126a}, we deduce that
$
|f(\x,u)|
\le\sup_{\x\in\Omega}|f(\x,0)|
+\left(\rho+\nicefrac{2}{\lambda}\right)|u|,
$
which is the required estimate; cf.~\cite[p.~9, $(p_2)$]{R86}.
\end{proof}

\begin{proof}[Proof of Theorem~\ref{thm:alternative}]
First notice that $\mathcal{E}=\{\E(u^n)\}_{n}$ is a bounded sequence. Indeed, from Proposition~\ref{pr:stab}, recalling that $\mathcal{E}$ is monotone decreasing, we infer that $\mathcal{E}$ is bounded from above. Moreover, from the weak coercivity of $\E$, we deduce that $\mathcal{E}$ is bounded from below, cf.~Remark~\ref{rem:37}; in addition, owing to Proposition~\ref{pr:res}, this observation implies that $\E'(u^n)\to0$ in $\VV^\prime$ as $n\to\infty$.

Combining the boundedness of $\mathcal{E}$ and the weak coercivity of $\E$, we see that the sequence $\{u^n\}_{n}$ is bounded in $\VV$. Then, due to Lemma~\ref{lem:mp}, we establish that $\{u^n\}_{n}$ has a (strongly) convergent subsequence $\{u^{n'}\}_{n'}$ in $\VV$, with a limit $u^\star\in\VV$, i.e.
\begin{equation*}
\lim_{n'\to\infty}\NN{u^\star-u^{n'}}_{\Delta t}=0.
\end{equation*}
Thus, for fixed $v\in\VV$, using Lemma~\ref{lem:aux2}, we note that
\[
\left|\dprod{\E'(u^\star)-\E'(u^{n'})}{v}\right|
\le L_{\E'}(\lambda) \NN{u^\star-u^{n'}}_{\lambda} \NN{v}_{\lambda}\to0,
\]
for $n'\to\infty$, with $L_{\E'}(\lambda)$ from~\eqref{eq:L}. 
Hence, involving Proposition~\ref{pr:res}, we conclude that
\[
\dprod{\E'(u^\star)}{v}
=\dprod{\E'(u^{n'})}{v}+\dprod{\E'(u^\star)-\E'(u^{n'})}{v}\to0,
\]
for $n'\to\infty$. It follows that $u^\star$ solves the weak formulation~\eqref{eq:poissonweak}. 
\end{proof}   

Now let us consider the situation where $\WW\subset\VV$ is a closed subspace.

\begin{theorem}[Convergence in closed subspaces]\label{thm:convW}
Let $\WW\subseteq\VV$ be a closed linear subspace, and $\Delta t>0$ as in Proposition~\ref{pr:stab}. If $\E$ is weakly coercive, cf.~\eqref{eq:wc}, then the sequence $\{u^n\}_{n}\subset\WW$ generated by the iteration~\eqref{eq:iterationW} has a subsequence $\{u^{n'}\}_{n'}$ that converges weakly in $\WW$ and strongly in $\L^2(\Omega)$ to a solution of~\eqref{eq:weakW}.
\end{theorem}

\begin{proof}
By the same argument as in the proof of Theorem~\ref{thm:alternative}, we note that $\{u^n\}_{n}$ is a bounded sequence in $\WW$. Since $\WW$ is a Hilbert (sub)space, and thus reflexive, there exists a subsequence $\{u^{n'}\}_{n'}$ and $u^\star\in\WW$ such that $u^{n'}\rightharpoonup u^\star$ in $\WW$ as $n'\to\infty$, i.e., in particular,
\[
\lim_{n'\to\infty}\int_\Omega\nabla(u^\star-u^{n'})\cdot\nabla w\dx=0\qquad\forall w\in\WW.
\] 
Furthermore, we note that $\WW\hookrightarrow \Lom{2}$ is a compact embedding (in dimensions $d=1,2,3$), i.e. there is a subsubsequence (for simplicity, still denoted by $\{u^{n'}\}_{n'}$) such that $u^{n'}\to u^\star$ strongly in $\Lom{2}$ as $n'\to\infty$. Hence, invoking the Lipschitz continuity~\eqref{eq:flipschitz}, we obtain
\[
\lim_{n'\to\infty}\left|\int_\Omega\left(f(\x,u^{n'})-f(\x,u^\star)\right)w\dx\right| \leq \left(\sigma_f(\Delta t)+\frac{1}{\Delta t}\right)\lim_{n'\to \infty} \norm{u^{n'}-u^\star}_{\L^2(\Omega)}\norm{w}_{\L^2(\Omega)}
=0,
\]
for all $w \in \WW$, and thus, thanks to Proposition~\ref{pr:res}, $\dprod{\E'(u^\star)}{w}=0$ for all $w\in\WW$.
\end{proof}

Finally, for the purpose of numerical approximations to be studied in \S\ref{sec:adaptiveIP}, we discuss the case where $\WW\subset\VV$ is a Galerkin subspace. The subsequent result is a straightforward consequence of the fact that weak and strong convergence are interchangeable in finite-dimensional spaces.

\begin{cor}[Convergence in discrete spaces]\label{cor:W}
If the linear subspace $\WW$ is finite-dimensional, i.e. $\dim\WW<\infty$, then the subsequence $\{u^{n'}\}_{n'}$ from Theorem~\ref{thm:convW} converges strongly in $\WW$ to a solution $u^\star$ of~\eqref{eq:weakW}.
\end{cor}

\section{Adaptive iterative linearised Galerkin method} \label{sec:adaptiveIP}

In this section, we will present and analyse an adaptive algorithm that exploits an interplay of the iterative linearisation procedure~\eqref{eq:iteration} and abstract adaptive Galerkin discretisations thereof. To this end, we consider a sequence of finite-dimensional Galerkin subspaces $\{\VV_N\}_N \subset \VV$ with  the  hierarchy  property $\VV_0 \subset \VV_1 \subset \VV_2 \subset \dotsc \subset \VV$. The sequence obtained by the iteration scheme~\eqref{eq:iterationW} on $\WW=\VV_N$ will be denoted by $\{u_N^n\}_{n}$. 


\subsection{Components of the adaptive procedure}

The proposed adaptive algorithm is based on three components, which we will outline in the sequel.

\subsubsection*{{\rm(i)} Iterative linearisation:}
On a given discrete subspace $\WW:=\VV_N \subset \VV$ we perform the iterative linearisation scheme~\eqref{eq:iterationW} until the discrete approximation $u_N^n\in\VV_N$ is close enough to a solution of the corresponding  discrete weak problem~\eqref{eq:weakW}. In order to evaluate the quality of a discrete approximation, for $\lambda > 0$ and $u\in\VV_N$, we introduce the discrete residual $\R^\lambda_N(u)\in \VV_N$ through the weak formulation
\begin{equation}\label{eq:RN}
\B_\lambda(\R^\lambda_N(u),v)=\dprod{\E'(u)}{v}\qquad \forall v\in \VV_N.
\end{equation}
We note that $u \in \VV_N$ is a solution of~\eqref{eq:weakW} if and only if $\R^\lambda_N(u)=0$ in $\VV_N$, cf.~\eqref{eq:eprimeprob}. 

We stop the iteration~\eqref{eq:iterationW} on the current Galerkin space $\VV_N$ as soon as $\NN{\R^\lambda_N(u_N^n)}_{\lambda}$ is deemed small enough; the final iteration number on the space $\VV_N$ is denoted by $n^\star=n^\star(N)$, and the corresponding approximation by
\[
\u{}:=u^{n^\star(N)}_N\in\VV_N,\qquad N\ge 0.
\]
In order to derive a stopping criterion that is easily computable, under Assumption~\ref{apt:A1}, for suitable $\rho>0$ and $\Delta t\in\G(\CC)$, we employ~\eqref{eq:20210831a} and \eqref{eq:RN} to observe that
\begin{align*}
\frac{1}{\Delta t}\B_{\Delta t}(\delta_N^n,v)=-\dprod{\E'(u_N^n)}{v}
=-\B_{\Delta t}(\R^{\Delta t}_N(u^n_N),v) \qquad \forall \, v \in \VV_N,
\end{align*}
where $\delta_N^n:=u_N^{n+1}-u_N^n$. By the coercivity of the bilinear form $\B_{\Delta t}$ this yields the representation
\begin{equation}\label{eq:R}
\R^{\Delta t}_N(u^n_N)=\frac{1}{\Delta t}(u_N^{n}-u_N^{n+1}).
\end{equation}
In particular, this identity shows that the residual $\R^{\Delta t}_N(u_N^n)$ is a \emph{linearisation indicator}, which allows to monitor the performance of the linearised iteration~\eqref{eq:iterationW} on the present Galerkin space~$\VV_N$.

\subsubsection*{{\rm(ii)} Decision between linearisation and discretisation:} 

Provided that the conditions of Proposition~\ref{pr:stab} hold for suitable $\rho>0$ and $\Delta t\in\G(\CC)$, we deduce from the bound~\eqref{eq:stab}, i.e.
\begin{align*}
\gamma_f(\Delta t)\NN{u^{n-1}_N-u^n_N}^2_{\Delta t}
\le\E(u^{n-1}_N)-\E(u^{n}_N),\qquad n\ge 0,
\end{align*}
that the linearisation indicator from (i) becomes relatively small once the energy decay on the current Galerkin space $\VV_N$ levels off, or, in other terms, if the overall energy decay on $\VV_N$, defined by
\begin{align} \label{eq:GFIincrement}
	\Phi^n_N(\Delta t):= \E(u_N^{0}) -  \E(u_N^{n}),\qquad n\ge 0,
\end{align}
is large in comparison to (the square of) the iteration update in each step.
We express this situation by a bound of the form
\begin{equation}\label{eq:alpha}
\frac{1}{\Delta t}\NN{u^{n-1}_N-u^n_N}_{\Delta t}
\le\alpha\Phi^n_N(\Delta t),
\end{equation}
for an appropriate constant $0<\alpha<1$ (independent of $N$ and $n$). If $\Phi^1_N(\Delta t)=0$, then we note that $u_N^1=u_N^0$ is a solution of~\eqref{eq:weakW}. Otherwise, if $\Phi^1_N(\Delta t)>0$, for given $N \in \mathbb{N}$, and the conditions of Proposition~\ref{pr:res} hold, then we remark that~\eqref{eq:alpha} is satisfied whenever $n$ is large enough; indeed, \eqref{eq:dualconvergence} states that
\[
\frac{1}{\Delta t}\NN{u^n_N-u^{n-1}_N}_{\Delta t}
=\NN{\R_N^{\Delta t}(u_N^{n-1})}_{\Delta t}=\sup_{v\in\VV_N\setminus\{0\}}\frac{\dprod{\E'(u_N^{n-1})}{v}}{\NN{v}_{\Delta t}} \to 0 \quad \text{as} \  n \to \infty,
\] 
whereas $\Phi^n_N(\Delta t)\geq \Phi^1_N(\Delta t)>0$ for all $n \geq 1$.

\subsubsection*{{\rm(iii)} Adaptive Galerkin space enrichments:}
The construction of the discrete Galerkin spaces is based on the assumption that we have at our disposal an adaptive strategy that is able to identify some local information on the error, and thereby, to appropriately enrich the current Galerkin space $\VV_N$ once the norm of the discrete residual is sufficiently small, cf.~(i).
To this end, for the iteration~\eqref{eq:iterationW} on the space $\VV_N$, we decompose the dual norm of the PDE residual associated to the weak form~\eqref{eq:poissonweak}, cf.~\eqref{eq:res}, into the two parts
\begin{align*}
\norm{\E'(u^n_N)}_{\VV'}
=\mathcal{E}^{\mathrm{it}}_N(u^n_N)+
\mathcal{E}^{\mathrm{dis}}_N(u^n_N),\qquad n\ge 0,
\end{align*}
where, upon involving~\eqref{eq:RN} and~\eqref{eq:R}, the term
\[
\mathcal{E}^{\mathrm{it}}_N(u^n_N):
=\norm{\E'(u^n_N)}_{\VV_N'}
=\NN{\R^{\Delta t}_N(u_N^n)}_{\Delta t}
=\frac{1}{\Delta t}\NN{u_N^n-u_{N}^{n+1}}_{\Delta t}
\]
takes the role of an iteration error on $\VV_N$, which vanishes for $n\to\infty$, cf.~Proposition~\ref{pr:res}, and
\begin{equation*}
\mathcal{E}^{\mathrm{dis}}_N(u^n_N):=\norm{\E'(u^n_N)}_{\VV'}-\norm{\E'(u^n_N)}_{\VV_N'}
\end{equation*}
is a discretisation part, with
\begin{equation} \label{eq:auxrem}
\lim_{n\to\infty}\left(\mathcal{E}^{\mathrm{dis}}_N(u^n_N)-\norm{\E'(u^n_N)}_{\VV'}\right)=0.
\end{equation}
We suppose that the space $\VV_N$ is refined in such a way that the discretisation part of the residual for the final approximation $\u{}=u_N^{n^\star(N)} \in \VV_N$ in a (hierarchically) enriched space $\VV_{N+1}$ is reduced by a uniform factor $0<q<1$, i.e.
\begin{align}\label{eq:disq}
\mathcal{E}^{\mathrm{dis}}_{N+1}(\u{})
=\mathcal{E}^{\mathrm{dis}}_{N+1}(u_{N+1}^0)
\le q\, \mathcal{E}^{\mathrm{dis}}_N(\u{});
\end{align}
the initial guess on the new subspace $u_{N+1}^0 \in\VV_{N+1}$ is defined through the canonical embedding of the final approximation $\u{}\in \VV_N\hookrightarrow\VV_{N+1}$.

\begin{remark}
Based on observing~\eqref{eq:auxrem}, for $n^\star(N)$ sufficiently large, the contraction assumption~\eqref{eq:disq} could be replaced by
\begin{equation} \label{eq:remarkcontraction}
\widetilde q\norm{\E'(\u{})}_{\VV'}\le\norm{\E'(\u{})}_{\VV_{N+1}'},
\end{equation}
with a constant $0<\widetilde q<1$. Indeed, if we imposed \eqref{eq:remarkcontraction} in place of~\eqref{eq:disq} in Theorem~\ref{thm:ILG} below, the statement would remain valid. Thereby, $\widetilde{q}$ would simply replace $(1-q)$ in the proof, cf.~\eqref{eq:replace}.
\end{remark}

\begin{algorithm}
\caption{Adaptive iterative linearised Galerkin algorithm}
\label{alg:two}
\begin{algorithmic}[1]
\State Prescribe the adaptivity parameter~$\alpha\in(0,1)$. 
\State Input an initial Galerkin space $\VV_0$ and initial guess $u^{0}_0 \in \VV_0$.
\State Set $N=n=0$, and choose $\Delta t \in \G(\CC)$, cf.~\eqref{eq:set}.
\While {true}
\Repeat 
\State Do one iterative step~\eqref{eq:iterationW} in $\VV_N$ to obtain $u_N^{n+1}$ from $u_N^n$.
\State Update $n \leftarrow n+1$.
\Until{$\NN{\R^{\Delta t}_N(u_N^n)}_{\Delta t} \leq \alpha\Phi_N^n(\Delta t)$}
\State Set $u_N^{n^\star}:=u_N^n$.
\State \multiline{Enrich the Galerkin space $\VV_N$ appropriately based on the local error indicators in order to obtain $\VV_{N+1}$.}
\State Define $u_{N+1}^0:=u_N^{n}$ by canonical embedding $\VV_N \hookrightarrow \VV_{N+1}$. 
\State Update $N \leftarrow N+1$ and $n \leftarrow 0$.
\EndWhile
\State \textsc{Return} the sequence $\{\u{}=u_N^{n^\star}\}_{N}$.
\end{algorithmic}
\end{algorithm}

\subsection{Convergence analysis}

Before we will state and prove a convergence result for our adaptive iterative linearised Galerkin algorithm, we first present an auxiliary result. 

\begin{lemma} \label{lem:aux1}
Given the assumptions of Proposition~\ref{pr:stab}. Let $\{\u{}\}_N \subset \VV$ be the sequence generated by Algorithm~\ref{alg:two} and assume that $\{\E(\u{})\}_N$ is bounded from below. Then, we have that
\begin{align} \label{eq:energydiff}
\left|\E(\u{})-\E(\up{})\right| \to 0,
\end{align}
and
\begin{align} \label{eq:disresconv}
\norm{\E'(\u{})}_{\VV_N'}=\NN{\R_N^{\Delta t}(\u{})}_{\Delta t} \to 0,
\end{align}
as $N\to\infty$.
\end{lemma}

\begin{proof}
From Proposition~\ref{pr:stab} recall that $\{\E(\u{})\}_N$ is a decreasing sequence in $\mathbb{R}$, which, in addition, is assumed to be bounded from below. Hence, it converges, and~\eqref{eq:energydiff} follows. Furthermore, by modus operandi of Algorithm~\ref{alg:two}, we have that
\begin{align*}
\alpha^{-1}\NN{\R_{N}^{\Delta t}(\u{})}_{\Delta t} 
\le\Phi_N^{n^\star(N)}(\Delta t)
=\E(u^0_N)-\E(\u{})
=\E(\um)-\E(\u{}) \to 0,
\end{align*}
as $N\to\infty$, where we have used~\eqref{eq:GFIincrement} and~\eqref{eq:energydiff}.
\end{proof}

We prove the ensuing convergence result. Here, as already elaborated earlier, we remark that the weak coercivity of $\E$, cf.~\eqref{eq:wc}, and the energy decay established in Proposition~\ref{pr:stab} imply the boundedness of the sequence $\{\u{}\}_N$ in $\VV$.

\begin{theorem}[Convergence of the adaptive ILG procedure]\label{thm:ILG}
Given the assumptions of Proposition~\ref{pr:res}. If the sequence of Galerkin spaces $\{\VV_N\}_N$ satisfies the property~\eqref{eq:disq} and the sequence of final approximations $\{\u{}\}_{N}$ generated by Algorithm~\ref{alg:two} is bounded in $\VV$, then $\{\u{}\}_{N}$ has a subsequence that converges weakly in $\VV$ and strongly in $\L^2(\Omega)$ to a solution of the weak formulation~\eqref{eq:poissonweak}. Moreover, if $f$ is continuous on $\overline{\Omega} \times \mathbb{R}$, then it even holds strong convergence in $\VV$ (for a further subsequence). 
\end{theorem}

\begin{proof}
By the very same arguments as in the proof of Theorem~\ref{thm:convW}, there exists a subsequence $\{\u{j}\}_j \subseteq \{\u{}\}_N$ and an element ${u^\star} \in \VV$ such that $\{\u{j}\}_j$ converges to ${u^\star}$ weakly in $\VV$ and strongly in $\L^2(\Omega)$. Furthermore, we have that
\begin{align} \label{eq:weakEconv}
\lim_{j \to \infty} \dprod{\E'(\u{j})}{v} =\dprod{\E'({u^\star})}{v} \qquad \forall v \in \VV.
\end{align}
We aim to show that $\E'({u^\star})=0$ in $\VV'$, i.e.~${u^\star}$ is a solution of the weak formulation~\eqref{eq:poissonweak}. For that purpose, by contradiction, we assume that there exists $\varepsilon>0$ such that 
$
\norm{\E'({u^\star})}_{\VV'}> \varepsilon;
$
equivalently, there is ${v^\star} \in \VV$ with $\NN{{v^\star}}_{\Delta t}=1$ such that
$
\dprod{\E'({u^\star})}{{v^\star}}> \varepsilon.
$
Hence, for all $j$ large enough, together with~\eqref{eq:weakEconv}, we have that
$
\norm{\E'(\u{j})}_{\VV'}
\ge \dprod{\E'(\u{j})}{{v^\star}}>\varepsilon
$
. Recalling the contraction property~\eqref{eq:disq}, i.e.
\[
\norm{\E'(\u{j})}_{\VV'}-\norm{\E'(\u{j})}_{\VV_{N_j+1}'}
\le q \Big(\norm{\E'(\u{j})}_{\VV'}-\norm{\E'(\u{j})}_{\VV_{N_j}'}\Big),
\]
this further yields
\begin{align} \label{eq:replace}
(1-q)\epsilon
< \norm{\E'(\u{j})}_{\VV_{N_j+1}'}-q\norm{\E'(\u{j})}_{\VV_{N_j}'}
\le\norm{\E'(\u{j})}_{\VV_{N_j+1}'}\le \norm{\E'(\u{j})}_{\VV_{N_{j+1}}'},
\end{align}
for all $j$ large enough. 
Then, 
using the triangle inequality and the Lipschitz continuity bound~\eqref{eq:lipschitz}, we obtain
\begin{align*}
(1-q)\epsilon
&<
\norm{\E'(\u{j})-\E'(\u{j+1})}_{\VV_{N_{j+1}}'}
+\norm{\E'(\u{j+1})}_{\VV'_{N_{j+1}}}\\
&\le L_{\E'}(\Delta t)\NN{\u{j}-\u{j+1}}_{\Delta t}
+\norm{\E'(\u{j+1})}_{\VV_{N_{j+1}}'}.
\end{align*}
In light of \eqref{eq:disresconv}, we infer that
\[
\lim_{j\to\infty}\norm{\E'(\u{j+1})}_{\VV'_{N_{j+1}}}=0.
\]
Hence, in order to derive a contradiction, it is sufficient to show that
\begin{align}\label{eq:suff}
\lim_{j\to\infty}\NN{\delta_j}_{\Delta t}=0,
\end{align}
where we let 
$
\delta_j:=\u{j}-\u{j+1}
$
(which, by assumption, is uniformly bounded for all $j$).
For this purpose, we recall Lemma~\ref{lem:EE} to note the identity
\begin{align}\label{eq:contr}
\frac{1}{2 \Delta t} \NN{\delta_j}_{\Delta t}^2
=
\E(\u{j})-\E(\u{j+1})
-\dprod{\E'(\u{j+1})}{\delta_j}
-\Psi_{\Delta t}(\u{j+1},\u{j}),
\end{align}
with
\begin{align*}
\left|\Psi_{\Delta t}(\u{j+1},\u{j})\right| \leq \frac{1}{2} \sigma_f(\Delta t) \norm{\delta_j}_{\L^2(\Omega)}^2.
\end{align*}
Due to the strong convergence in $\L^2(\Omega)$, we infer that 
$
\lim_{j\to\infty}\left|\Psi_{\Delta t}(\u{j+1},\u{j})\right|=0.
$
Furthermore, by proceeding along the lines of the proof of Lemma~\ref{lem:aux1} we obtain that
\begin{align*}
\lim_{j\to\infty}\left(\E(\u{j+1})-\E(\u{j})\right)=0\qquad\text{and}\qquad
\lim_{j\to\infty}\dprod{\E'(\u{j+1})}{\delta_j}=0.
\end{align*}
Hence, from~\eqref{eq:contr} we deduce~\eqref{eq:suff} and, thereby, the required contradiction. Finally, if $f$ is continuous on $\overline{\Omega} \times \mathbb{R}$, then we apply Lemma~\ref{lem:mp}, which provides a strongly convergent subsequence of~$\{\u{j}\}_{j}$. 
\end{proof}

\subsection{Energy contraction} \label{sec:energycontraction}

We present an alternative condition to the space enrichment assumption~\eqref{eq:disq} that is solely based on an energy reduction property, however, requires some stronger assumptions for the convergence analysis. To this end, we note first that if $\mathcal{U} \subset \VV$ is a non-empty, bounded, closed, convex subset such that $\E'|_{\mathcal{U}}$ is strongly monotone, then, for any closed subspace $\WW \subset \VV$, the energy $\E$ has a unique local minimiser $u_{\WW}^\star\in\WW \cap \mathcal{U}$; see, e.g., to~\cite[Ch.~25.5]{Zeidler:90}.

\begin{theorem} \label{thm:ILG2}
Given the assumptions of Proposition~\ref{pr:res}. Suppose that the sequence $\{\u{}\}_N$ generated by Algorithm~\ref{alg:two} is contained in a non-empty, bounded, closed, and convex subset $\mathcal{U} \subset \VV$  such that $\E'|_{\mathcal{U}}$ is strongly monotone, i.e. there exists a constant $\nu(\Delta t)>0$ with
\begin{align} \label{eq:strongmonotonicity}
\dprod{\E'(u)-\E'(v)}{u-v} \geq \nu(\Delta t) \NN{u-v}_{\Delta t}^2 \qquad \forall u,v \in \mathcal{U}.
\end{align}
If, moreover, the hierarchically refined Galerkin spaces $\{\VV_N\}_N$ resulting from Algorithm~1 satisfy the energy contraction property
\begin{align} \label{eq:spaceenergycontraction}
\E(u^\star_{\VV_{N+1}})-\E(u^\star_{\VV}) \leq q \left(\E(u^\star_{\VV_{N} })-\E(u^\star_{\VV }) \right) \qquad \forall N \geq 0,
\end{align}
where $q \in (0,1)$, and $u_{\VV}^\star, u_{\VV_N}^\star \in \mathrm{int}(\mathcal{U})$ for all $N$ large enough, then $\{\u{}\}_N$ converges strongly to $u^\star_{\VV}$, which is a weak solution of~\eqref{eq:poissonweak}.
\end{theorem}


\begin{proof}
For any closed subspace $\WW\subset\VV$, due to the strong monotonicity~\eqref{eq:strongmonotonicity} and the Lipschitz continuity~\eqref{eq:lipschitz}, we notice the equivalence 
\begin{align} \label{eq:normequivalence}
\frac{1}{2} \nu(\Delta t)\NN{u-u_{\WW}^\star}_{\Delta t}^2 \leq \E(u)-\E(u^\star_{\WW}) \leq \frac{1}{2}L_{\E'}(\Delta t) \NN{u-u_{\WW}^\star}_{\Delta t}^2 \qquad \forall u \in \WW \cap \mathcal{U},
\end{align} 
cf.~\cite[Lem.~2]{HeidWihler2:19v1}.
Since the sequence of Galerkin spaces is hierarchical, this immediately implies that $\{\E(u^\star_{\VV_{N}})\}_N$ is a monotone decreasing sequence, which is bounded from below by $\E(u^\star_{\VV})$. Therefore, 
\[
\E^\star:=\lim_{N\to\infty}\E(u^\star_{\VV_N})\geq \E(u^\star_{\VV})
\]
exists. We claim that $\E^\star=\E(u^\star_{\VV})$; if not, then there exists $\epsilon>0$ such that 
\begin{align*}
\E(u^\star_{\VV_{N}}) \geq \E^\star \ge \E(u^\star_{\VV})+\epsilon,  \qquad N \geq 0,
\end{align*}
which, in view of~\eqref{eq:spaceenergycontraction}, leads to
\begin{align*}
0<\epsilon (1-q)\leq
\E(u^\star_{\VV_{N}})-\E(u^\star_{\VV_{N+1}})\xrightarrow{N\to\infty} 0,
\end{align*}
a contradiction. Recalling~\eqref{eq:normequivalence}, it follows that
\begin{align*}
0 \leq \NN{u_{\VV_N}^\star-u_{\VV}^\star}^2_{\Delta t} \leq \frac{2}{\nu(\Delta t)} \left(\E(u_{\VV_N}^\star)-\E(u^\star_{\VV})\right) \xrightarrow{N\to\infty} 0,
\end{align*}
i.e., $\{u_{\VV_N}^\star\}_N$ converges strongly to $u_{\VV}^\star$. Hence, by virtue of~\eqref{eq:strongmonotonicity}, we have
\begin{align*}
\nu(\Delta t) \NN{\u{}-u_{\VV_N}^\star}_{\Delta t}^2 
&\leq \dprod{\E'(\u{})-\E'(u^\star_{\VV_N})}{\u{}-u_{\VV_N}^\star}\\
&\leq \left(\norm{\E'(\u{})}_{\VV_N'}+\norm{\E'(u^\star_{\VV_N})}_{\VV_N'}\right) \NN{\u{}-u_{\VV_N}^\star}_{\Delta t},
\end{align*}
i.e.
\[
\NN{\u{}-u_{\VV_N}^\star}_{\Delta t} \leq \frac{1}{\nu(\Delta t) }\ \left(\norm{\E'(\u{})}_{\VV_N'}+\norm{\E'(u^\star_{\VV_N})}_{\VV_N'}\right).
\]
Upon exploiting~\eqref{eq:disresconv}, we notice that $\norm{\E'(\u{})}_{\VV_N'}\to0$ as $N\to\infty$. Moreover, for $N$ sufficiently large, because $u^\star_{\VV_N}$ is a local minimiser of $\E$ on $\VV_N \cap\mathrm{int}(\mathcal{U})$, we have $\E'(u^\star_{\VV_N})=0$ in $\VV_N'$. Hence, we conclude that
\begin{align*} 
\NN{\u{}-u^\star_{\VV}}_{\Delta t}
\le
\NN{\u{}-u_{\VV_N}^\star}_{\Delta t}+\NN{u_{\VV_N}^\star-u^\star_{\VV}}_{\Delta t}
\xrightarrow{N\to\infty} 0,
\end{align*}
for $N\to\infty$, which completes the proof.
\end{proof} 

\begin{remark} \label{rem:spaceconctraction}
Theorem~\ref{thm:ILG2} remains valid, if we replace the assumption~\eqref{eq:spaceenergycontraction} by 
\begin{align*} 
\E(\up{})-\E(u^\star_{\VV}) \leq q \left(\E(\u{})-\E(u^\star_{\VV }) \right) \qquad \forall N \geq 0,
\end{align*}
where $q \in (0,1)$ is fixed.
\end{remark}

\section{Energy-based adaptive finite element method}\label{sec:Energy-FEM}

In this section, given the variational setting and the analysis derived in~\S\ref{sec:energy_minimisation} (see, in particular, Proposition~\ref{pr:stab}), we utilise an adaptive mesh refinement strategy that is solely based on local energy reductions. To this end, we follow the recent approach presented in~\cite[\S 3]{HeidStammWihler:19}, which does not involve any a posteriori error indicators. 

\subsection{Finite element meshes and spaces}
   
For $N\ge0$, let $ \mathcal{T}_N=\{\kappa\}_{\kappa\in\mathcal{T}_N}$ be a regular and shape-regular mesh partition of $\Omega$ into disjoint open simplices. We consider the conforming finite element space
\begin{align} \label{eq:femspace}
\VV_{N}:=\{\varphi\in \H^1_0(\Omega):\,\varphi|_{\kappa} \in \mathbb{P}_{p}(\kappa) \ \forall \kappa \in \mathcal{T}_{N}\},
\end{align}
where $\mathbb{P}_{p}(\kappa)$ signifies the set of all polynomials of degree at most $p\in\mathbb{N}$ on $\kappa$. 

For the purpose of local refinements, for any element~$\kappa\in \mathcal{T}_N$, we consider the open patch~$\P_\kappa$ comprising of $\kappa$ and its immediate face-wise neighbours in the mesh~$\mathcal{T}_N$. Moreover, we define the modified patch~$\Pref_\kappa$ by uniformly (red) refining the element $\kappa$ into a (fixed) number of subelements; here, we assume that the introduction of any hanging nodes in $\P_\kappa$ is removed by introducing suitable (e.g.~green) refinements, see Figure~\ref{fig:ref}. 
%
We also introduce the associated (low-dimensional) space
\begin{equation}\label{eq:Vloc}
\VV(\Pref_\kappa):=\{v \in \VV: v|_\tau \in \mathbb{P}_p(\tau) \ \forall \tau \in \Pref_\kappa \ \text{and} \ v|_{\Omega \setminus \Pref_\kappa}=0\},
\end{equation}
which consists of all finite element functions that are locally supported on the patch~$\Pref_\kappa$.

\begin{figure}
\begin{center}
\begin{tabular}{cc}
\includegraphics[scale=0.2]{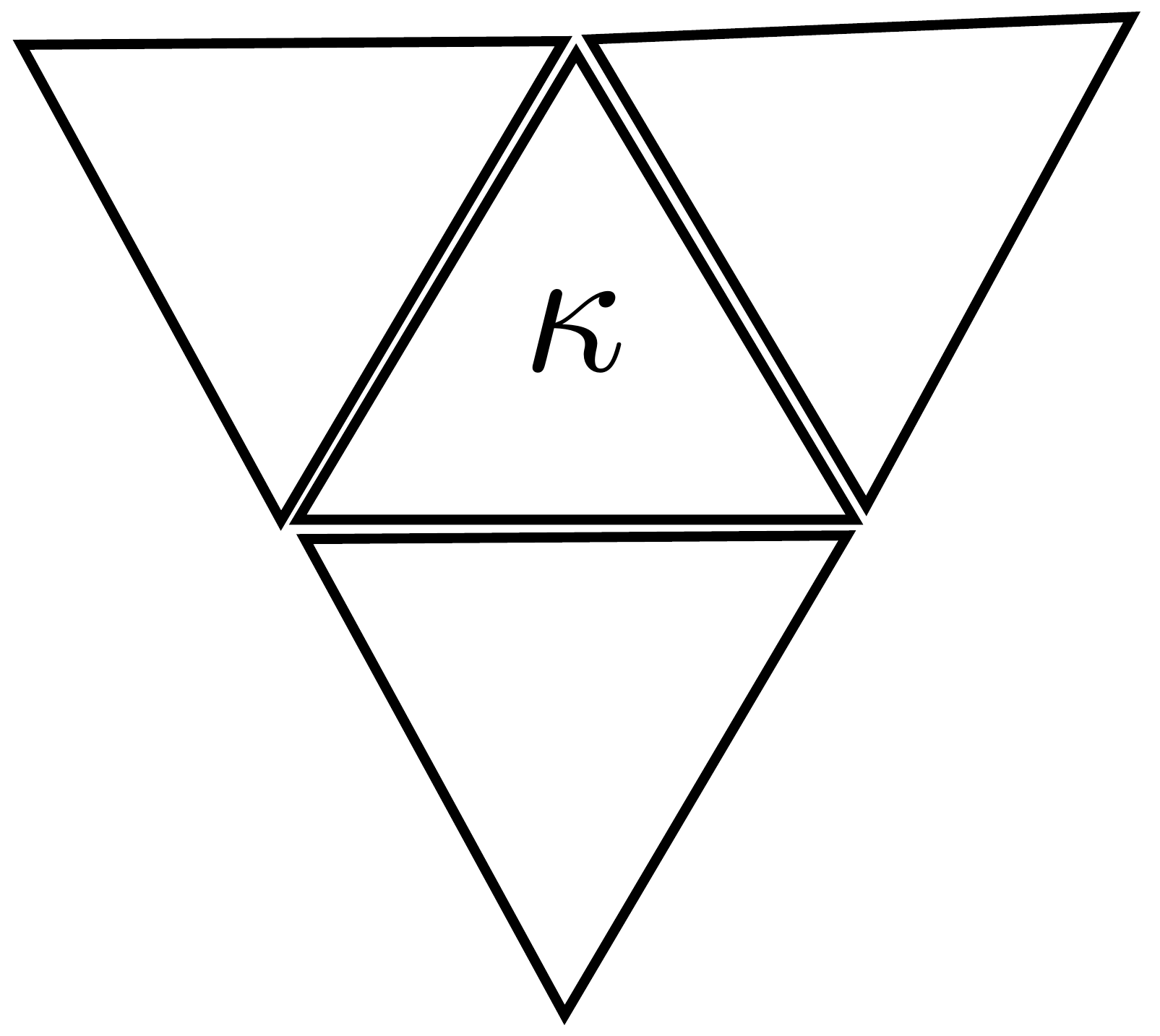} &
\includegraphics[scale=0.2]{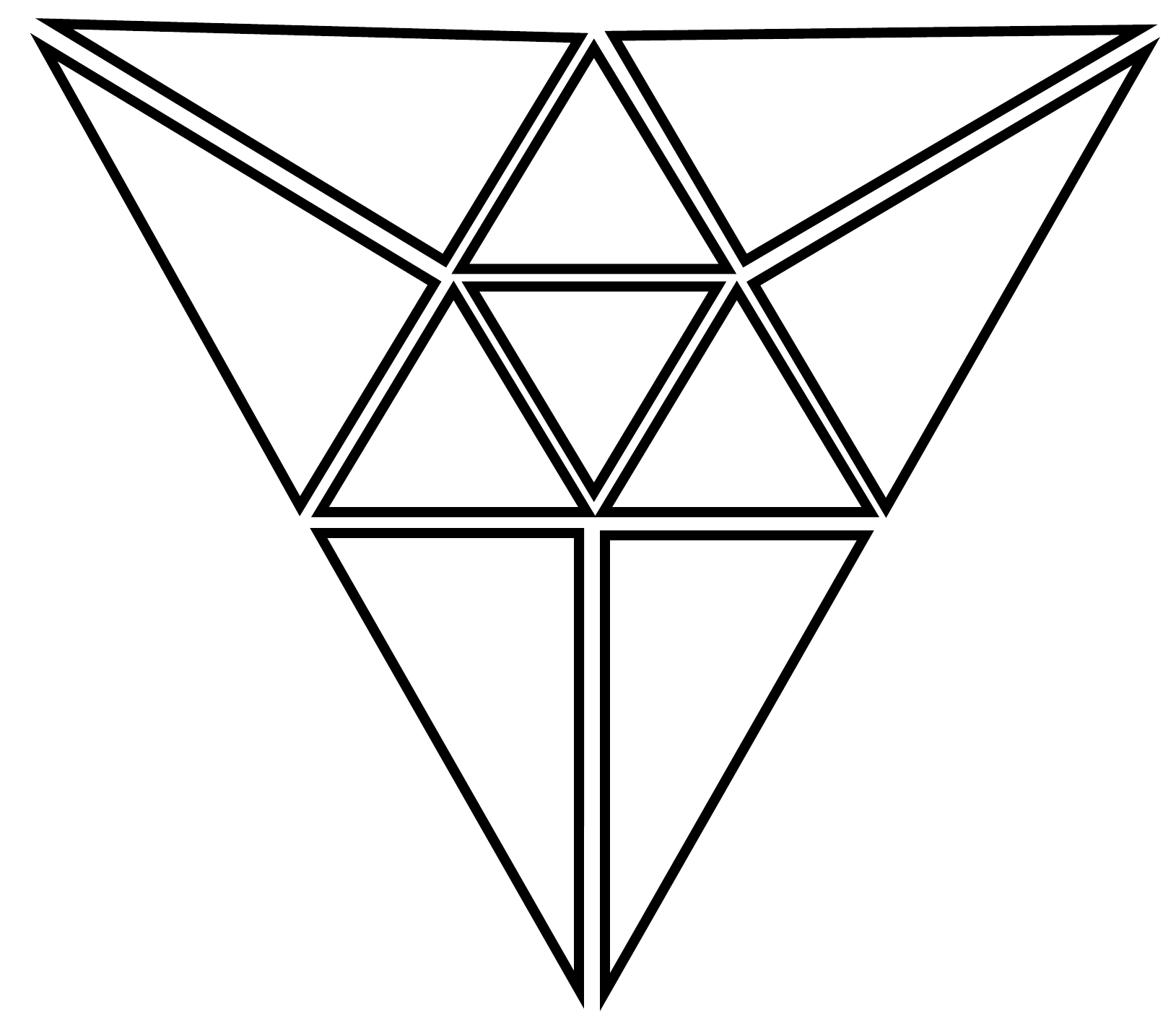}
\end{tabular}
\caption{Local element patches associated to a triangular element~$\kappa$. Left: Mesh patch~$\P_\kappa$ consisting of the element $\kappa$ and its face-neighbours. Right: Modified patch $\Pref_\kappa$ constructed based on red-refining $\kappa$ and on green-refining its neighbours.}
\label{fig:ref}
\end{center}
\end{figure}

\subsection{Energy-driven adaptive mesh refinement strategy}

Suppose that we have found a sufficiently accurate approximation $u_N^n \in \VV_N$ of the discrete problem~\eqref{eq:weakW}, with $\WW=\VV_N$, for some large enough~$n\ge 0$. Then, for each element $\kappa\in\mathcal{T}_N$, given a basis~$\{\xi^1_\kappa,\ldots,\xi^{m_\kappa}_\kappa\}$ of the local space $\VV(\Pref_\kappa)$ from~\eqref{eq:Vloc}, we introduce the extended space
\[
\Vh(\Pref_\kappa;u_N^n):=\Span\{\xi^1_\kappa,\ldots,\xi^{m_{\kappa}}_\kappa,u_N^n\}.
\] 
Now, upon performing one \emph{local} discrete iteration step~\eqref{eq:iterationW} on $\WW:=\Vh(\Pref_\kappa;u_N^n) \subset \VV$ we obtain a potentially improved approximation $\widetilde u_{N,\kappa}^n \in \Vh(\Pref_\kappa;u_N^n)$ through
\begin{equation}\label{eq:locweak}
\B_{\Delta t}(\widetilde u_{N,\kappa}^n,v)=\ell_{\Delta t}(u_N^n;v)
\qquad\forall v\in\Vh(\Pref_\kappa;u_N^n).
\end{equation}
Here, we emphasise that the discrete iteration step~\eqref{eq:iterationW} on the low-dimensional space $\WW=\Vh(\Pref_\kappa;u_N^n)$ entails hardly any computational cost; for instance, for dimension $d=2$ and polynomial degree~$p=1$, the dimension of the locally refined space $\Vh(\Pref_\kappa;u_N^n)$ is typically 3 or 4. Moreover, the local iteration steps can be performed individually, and thus in parallel, for each element $\kappa\in\mathcal{T}_N$.

\begin{lemma}
Suppose that the assumptions of Proposition~\ref{pr:stab} are satisfied, and let $\kappa \in \mathcal{T}_N$. Then, for the local weak formulation~\eqref{eq:locweak} it holds that
\begin{align} \label{eq:localenergydecay}
 \Delta \E_N^n(\kappa):=\E(u_N^n)-\E(\widetilde u_{N,\kappa}^n) \ge  0,
\end{align}
for each $n\ge 0$.
\end{lemma}

\begin{proof}
We define the subspace $\WW=\Vh(\Pref_\kappa;u_N^n)\subset\VV$. Then, the claim follows immediately from Proposition~\ref{pr:stab} upon observing that~\eqref{eq:locweak} corresponds to the iteration~\eqref{eq:iterationW} with $u^0=u_N^n$ and $u^1=\widetilde u_{N,\kappa}^n$.
\end{proof}

\begin{algorithm}
\caption{Energy-based adaptive mesh refinement}
\label{alg:refen}
\begin{algorithmic}[1]
\State Prescribe a mesh refinement parameter~$\theta\in(0,1)$. 
\State Input a finite element mesh~$\mathcal{T}_N$, and a finite element function $u^{n}_N \in \VV_N$.
\For {all elements $\kappa\in\mathcal{T}_N$}
	\State\multiline{\textsc{Solve} one local discrete iteration step~\eqref{eq:locweak} in the low-dimensional space $\Vh(\Pref_\kappa;u_N^n)$ to obtain a potentially improved local approximation~$\widetilde u_{N,\kappa}^n$.}  
	\State \textsc{Compute} the local energy decay~$\Delta \E_N^n(\kappa)$ from~\eqref{eq:localenergydecay}.    
\EndFor
\State \textsc{Mark} a subset ~$\mathcal{K} \subset \mathcal{T}_N$ of minimal cardinality which fulfils the D\"orfler marking criterion
\[
\sum_{\kappa \in \mathcal{K}} \Delta \E_N^n(\kappa) \geq \theta \sum_{\kappa \in \mathcal{T}_N} \Delta \E_N^n(\kappa).
\]
\State \textsc{Refine} all elements in~$\mathcal{K}$ for the sake of generating a new mesh~$\mathcal{T}_{N+1}$.
\end{algorithmic}
\end{algorithm}

The value $\Delta \E_N^n(\kappa)$ from~\eqref{eq:localenergydecay} indicates the potential energy reduction due to a refinement of the element $\kappa$. This observation motivates the energy-based adaptive mesh refinement procedure outlined in Algorithm~\ref{alg:refen}.   
From a practical point of view, we remark that the evaluation of~\eqref{eq:localenergydecay} requires a global integration for any element $\kappa \in \mathcal{T}_N$; this is computationally expensive if the dimension of the underlying finite element space is large. A possible remedy is to employ the energy expansion from Lemma~\ref{lem:EE} about $u_N^n$, which yields
\begin{align*}
 -\Delta \E_N^n(\kappa)=\E(\widetilde u_{N,\kappa}^n)-\E(u_N^n)\approx \dprod{\E'(u_N^n)}{\widetilde u_{N,\kappa}^n-u_N^n}.
\end{align*}  
Then, noting the unique linear combination
\[
\widetilde u_{N,\kappa}^n=\alpha u_N^n+\eta^n_{N,\kappa},
\]
with appropriate $\alpha\in\mathbb{R}$ and $\eta^n_{N,\kappa}\in\VV(\Pref_\kappa)$, and recalling~\eqref{eq:Eid} as well as~\eqref{eq:locweak}, we see that
\begin{align*}
\dprod{\E'(u_N^n)}{\widetilde u_{N,\kappa}^n-u_N^n}
&=\frac{1}{\Delta t}\left(\B_{\Delta t}(u_N^n,\widetilde u_{N,\kappa}^n-u_N^n)-\ell_{\Delta t}(u_N^n;\widetilde u_{N,\kappa}^n-u_N^n)\right)\\
&=\frac{1}{\Delta t}\B_{\Delta t}(u_N^n-\widetilde u_{N,\kappa}^n,\widetilde u_{N,\kappa}^n-u_N^n)\\
&=-\frac{1}{\Delta t}
\left(
(\alpha-1)^2\NN{u_N^n}^2_{\Delta t}
+\NN{\eta_{N,\kappa}^n}^2_{\Delta t}
+2(\alpha-1)\B_{\Delta t}(u_N^n,\eta_{N,\kappa}^n)
\right).
\end{align*}
In the above approximation, we observe that there is only one global integration, namely for the term $\NN{u_N^n}_{\Delta t}$, which is the same for all elements; the remaining two terms involve the locally supported function $\eta^n_{N,\kappa}$ on the patch $\omega_\kappa$. For more details concerning the computational complexity of Algorithm~\ref{alg:refen} we refer to~\cite[\S 3.5]{HeidStammWihler:19}, albeit the setting in the current article is slightly different.   

\subsection{Numerical experiments}\label{sec:numerics}

We will now perform some numerical tests in two spatial dimensions, with  Cartesian  coordinates  denoted  by $\x=(x,y)\in\mathbb{R}^2$. The finite element spaces consist of elementwise affine functions, i.e., we let $p=1$ in~\eqref{eq:femspace} and~\eqref{eq:Vloc}. In  all  examples,  we set $\alpha=\theta=\nicefrac12$ and employ the newest vertex bisection method~\cite{Mitchell:91} for the refinement in line~8 of Algorithm~\ref{alg:refen}, which yields shape-regular locally refined meshes. All our computations are initiated on a uniform and coarse triangulation of $\Omega$, and the starting guess is chosen as~$u_0^0 \equiv 0$.

\subsubsection{Convergence of the error} We begin by running an experiment with a manufactured solution.
 
\experiment \label{exp:SineGordon}
We consider the sine-Gordon type problem
\begin{subequations}\label{eq:ex1}
\begin{alignat}{2}
-\Delta u &= - \sin(u) -u+g(\x)  &\qquad& \text{in } \Omega,  \\
u&=0  && \text{on } \partial \Omega,  
\end{alignat}
\end{subequations}
i.e., the model~\eqref{eq:poisson} with the reaction term $f(\x,u)=-\sin(u)-u+g(\x)$, on the square domain $\Omega:=(0,1)^2$, where the function
\[
g(\x)=\sin(\pi x)\sin(\pi y)+2 \pi^2 \sin(\pi x) \sin(\pi y)+\sin(\sin(\pi x) \sin(\pi y))
\]
is constructed in such a way that $u^\star(\x)=\sin(\pi x) \sin(\pi y)$ solves~\eqref{eq:ex1}. For any $0<\lambda \leq 1$, we note that
$
\sigma_f(\lambda)=\sup_{u \in \mathbb{R}} \left|-\cos(u)-1+\nicefrac{1}{\lambda}\right|=\nicefrac{1}{\lambda},
$
whence it follows that $ \lambda \in\G(\rho)$ for any $\rho>0$ in Assumption~\ref{apt:A1}. In particular, due to Theorem~\ref{thm:convergence}, we conclude that the solution of~\eqref{eq:ex1} is unique, and that we can choose any step size $0<\Delta t \leq 1$. Rather than selecting the supposedly obvious choice $\Delta t=1$, we will set $\Delta t= \nicefrac{1}{2}$, which leads to a better performance of our algorithm for the given example.

In Figure~\ref{fig:SineGordonEx} (left) we plot the error $\errH:=\norm{\nabla(u^\star-\u{})}_{\L^2(\Omega)}$ against the number of degrees of freedom (henceforth denoted  by $\dof$), and, in addition, the number of iteration steps on each of the discrete spaces. We observe an (almost) optimal decay of the error of order $\mathcal{O}(\dof^{-\nicefrac12})$, whereby one iterative linearisation step is sufficient on each discrete space. Moreover, since we have at our disposal the global minimiser $u^\star$ of the underlying energy functional $\E$, we will further plot the energy contraction ratio
\begin{align} \label{eq:QN}
Q_N:=\frac{\E(\up{})-\E(u^\star)}{\E(\u{})-\E(u^\star)}, \qquad N \geq 0,
\end{align}
which is closely related to the factor $q$ from~\eqref{eq:spaceenergycontraction}, cf.~Remark~\ref{rem:spaceconctraction}. We note that $\E':\VV \to \VV'$ is strongly monotone in the experiment under consideration, and that $Q_N$ stabilises after an initial phase close to the value~$0.6$, see Figure~\ref{fig:SineGordonEx} (right).

\begin{figure}
\hfill
	\includegraphics[width=0.49\textwidth]{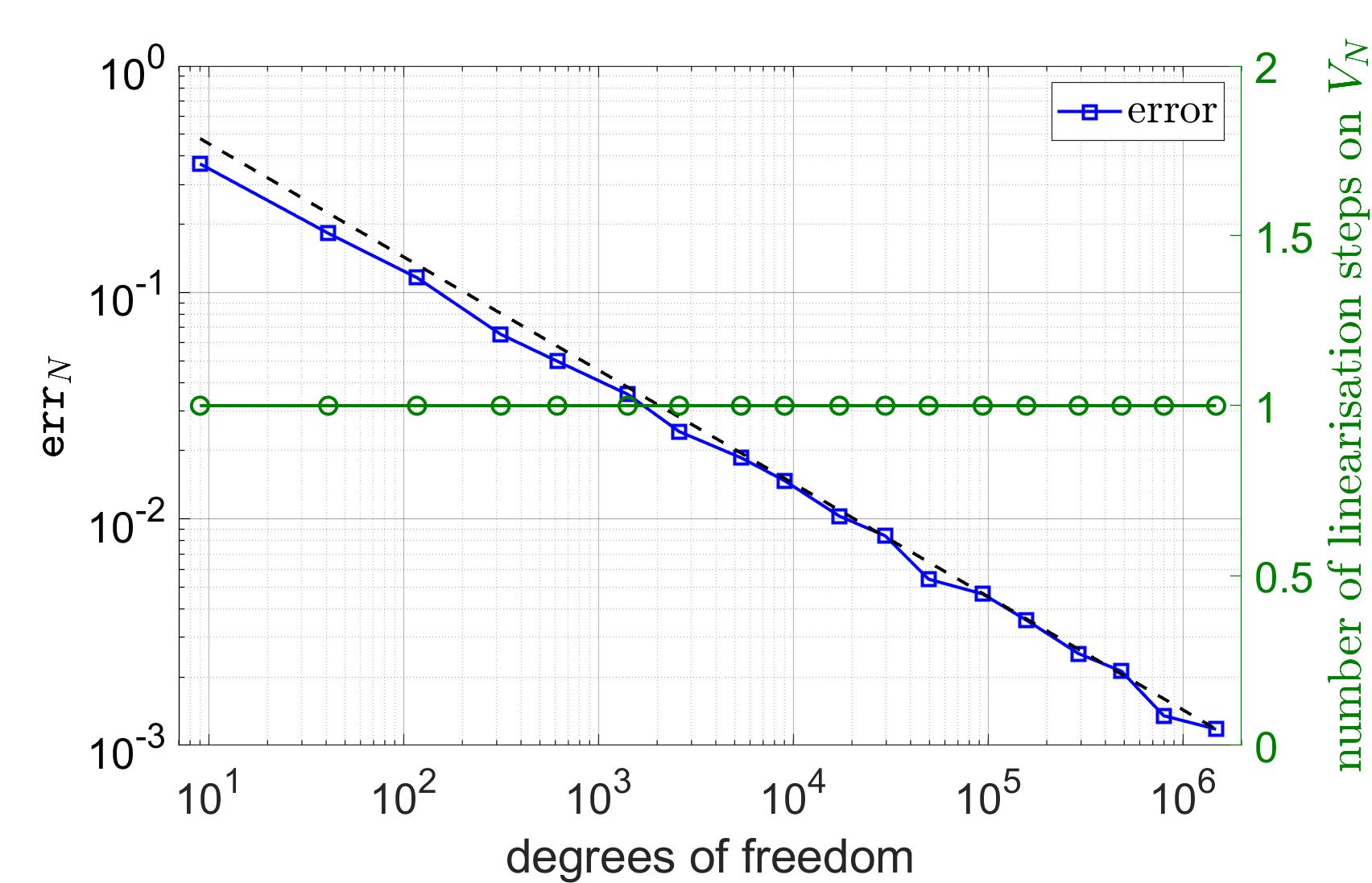}
	\hfill
	\includegraphics[width=0.49\textwidth]{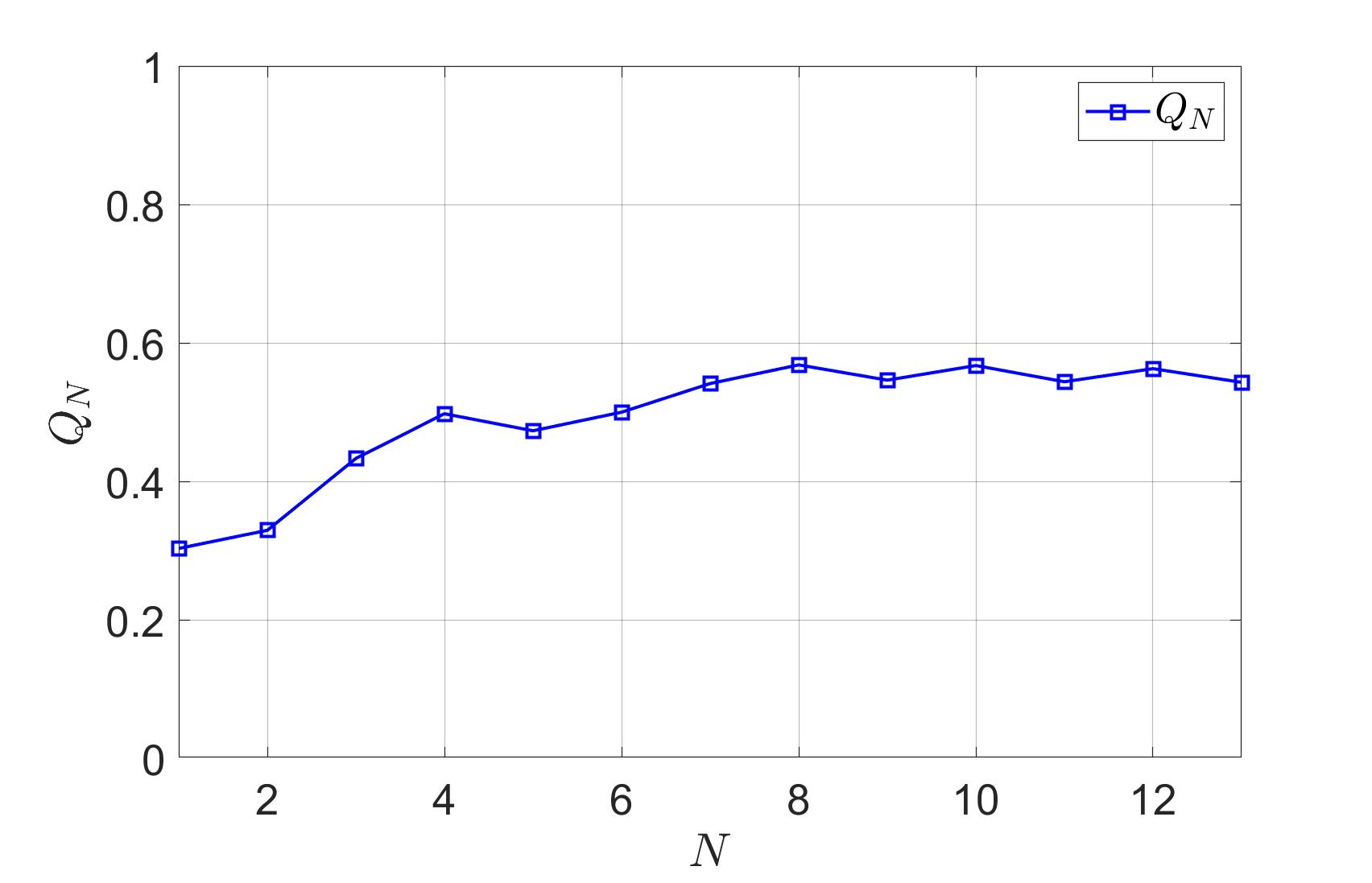}
	\hfill
	\caption{Experiment~\ref{exp:SineGordon}. Left: Convergence plot for the error, whereas the black dashed line indicates the optimal convergence rate of $\mathcal{O}(\dof^{-\nicefrac12})$. Right: Plot of the energy contraction factor $Q_N$.}
	\label{fig:SineGordonEx}
\end{figure}

\subsubsection{Convergence of the energy}\label{sc:ce}
In the experiments below, an analytical expression for a solution is no longer available. In order to still be able to perform a numerical study, we first apply our adaptive procedure for the purpose of computing a reference energy value based on $\mathcal{O}(10^6)$ degrees of freedom. Subsequently, we rerun the algorithm until the number of degrees of freedom exceeds $10^5$ (unless stated differently). We monitor the energy approximation by displaying the difference of the reference energy value and the energy value on each given discrete space. On a side note, we remark that we have also examined the decay of a standard residual estimator with respect to $\dof$, see, e.g.,~\cite{Verfurth:13}, whereby we have observed an (almost) optimal rate of order~$\mathcal{O}(\dof^{-\nicefrac12})$ (not included here).

\experiment \label{exp:SineGordonEnergy}
We consider a singularly perturbed sine-Gordon type equation
\begin{equation*}
\begin{aligned}
-\varepsilon \Delta u &= - \sin(u) -u+1  \quad&& \text{in } \Omega, \\
u&=0 \quad && \text{on } \partial \Omega,  
\end{aligned}
\end{equation*}
with $\Omega:=(0,1)^2$ and $0<\varepsilon\ll1$. Even though this problem is strongly perturbed, it fits into the framework of our analysis. Indeed, after dividing the partial differential equation by $\varepsilon$, we let $f(u)=\varepsilon^{-1}(-\sin(u)-u+1)$, and consider $0<\lambda \leq \varepsilon$ in~\eqref{eq:sigma}; for this range of $\lambda$ it holds
$
\sigma_f(\lambda)
=\nicefrac{1}{\lambda},
$
which shows that Assumption~\ref{apt:A1} is satisfied for any $\rho>0$. Hence, the situation is similar as in the previous Example~\ref{exp:SineGordon}, in particular, Theorem~\ref{thm:convergence} is applicable for the step size $\Delta t=\nicefrac{\varepsilon}{2}$. We experiment with the singular perturbation parameter $\varepsilon=10^{-5}$, which leads to extremely sharp boundary layers along~$\partial\Omega$; see Figure~\ref{fig:SineGordon} (left). Nonetheless, the energy error decays, after an initial phase of reduced convergence (during which the boundary layers are detected), at an almost optimal rate of $\mathcal{O}(\dof^{-1})$ as illustrated in Figure~\ref{fig:SineGordonQuotient} (left). This is due to the adaptive mesh refinement strategy, which properly detects the boundary layers, see Figure~\ref{fig:SineGordon} (right).
\begin{figure}
	\hfill
	\includegraphics[width=0.49\textwidth]{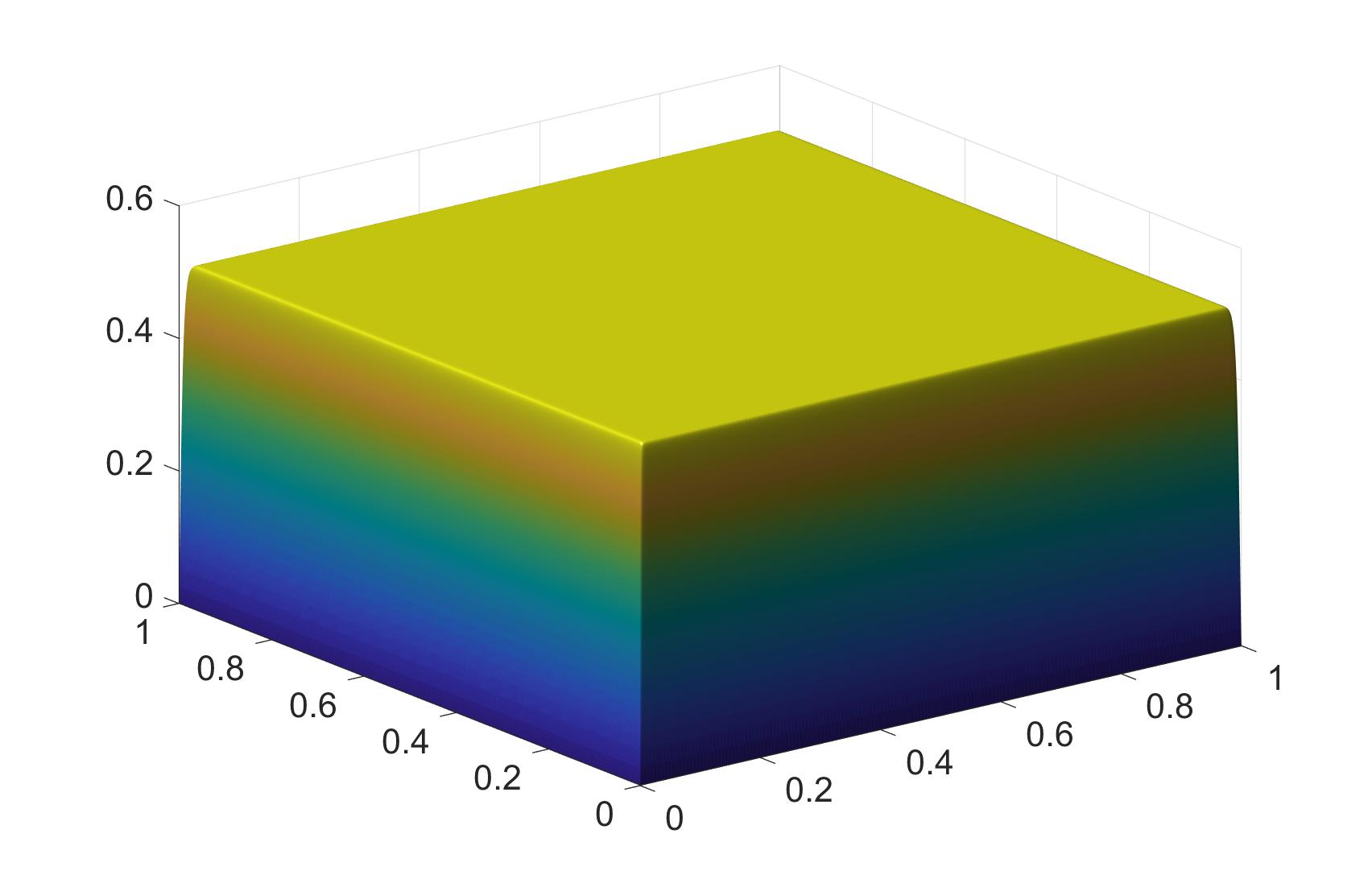}
	\hfill
	\includegraphics[width=0.49\textwidth]{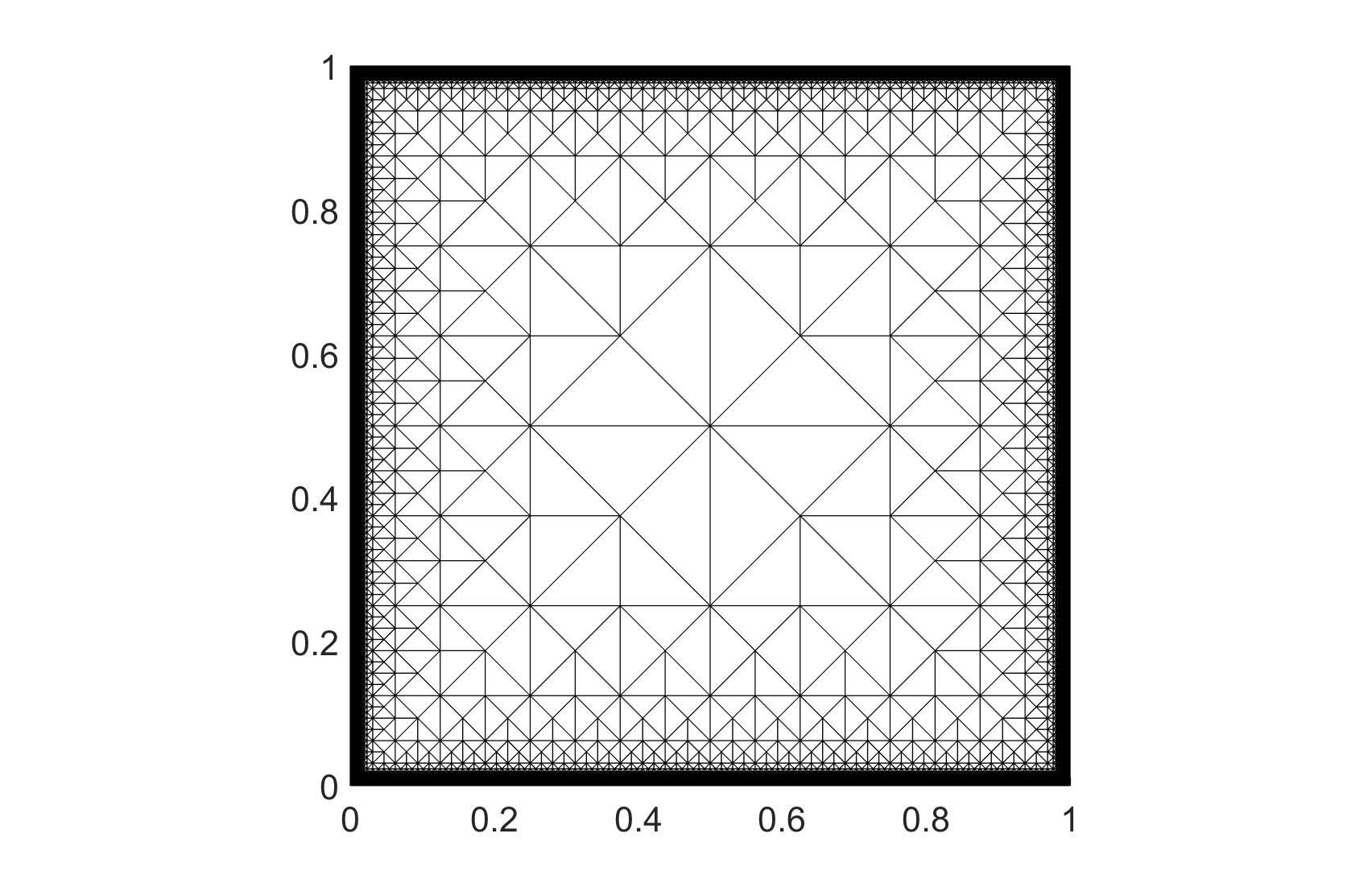}
	\hfill
	\caption{Experiment~\ref{exp:SineGordonEnergy}. Left: Approximated solution. Right: Adaptively refined mesh.}
	\label{fig:SineGordon}
\end{figure}
Furthermore, here and the following experiments, we depict the quotient $Q_N$, cf.~\eqref{eq:QN}, against the number of mesh refinements, whereby we use the computed reference energy value in place of $\E(u_{\VV}^\star)$. As we can observe in Figure~\ref{fig:SineGordonQuotient} (right), the quotient $Q_N$ stays uniformly away from 1; in particular, this contraction factor is bounded from above by $0.7$.

\begin{figure}
\hfill
\includegraphics[width=0.49 \textwidth]{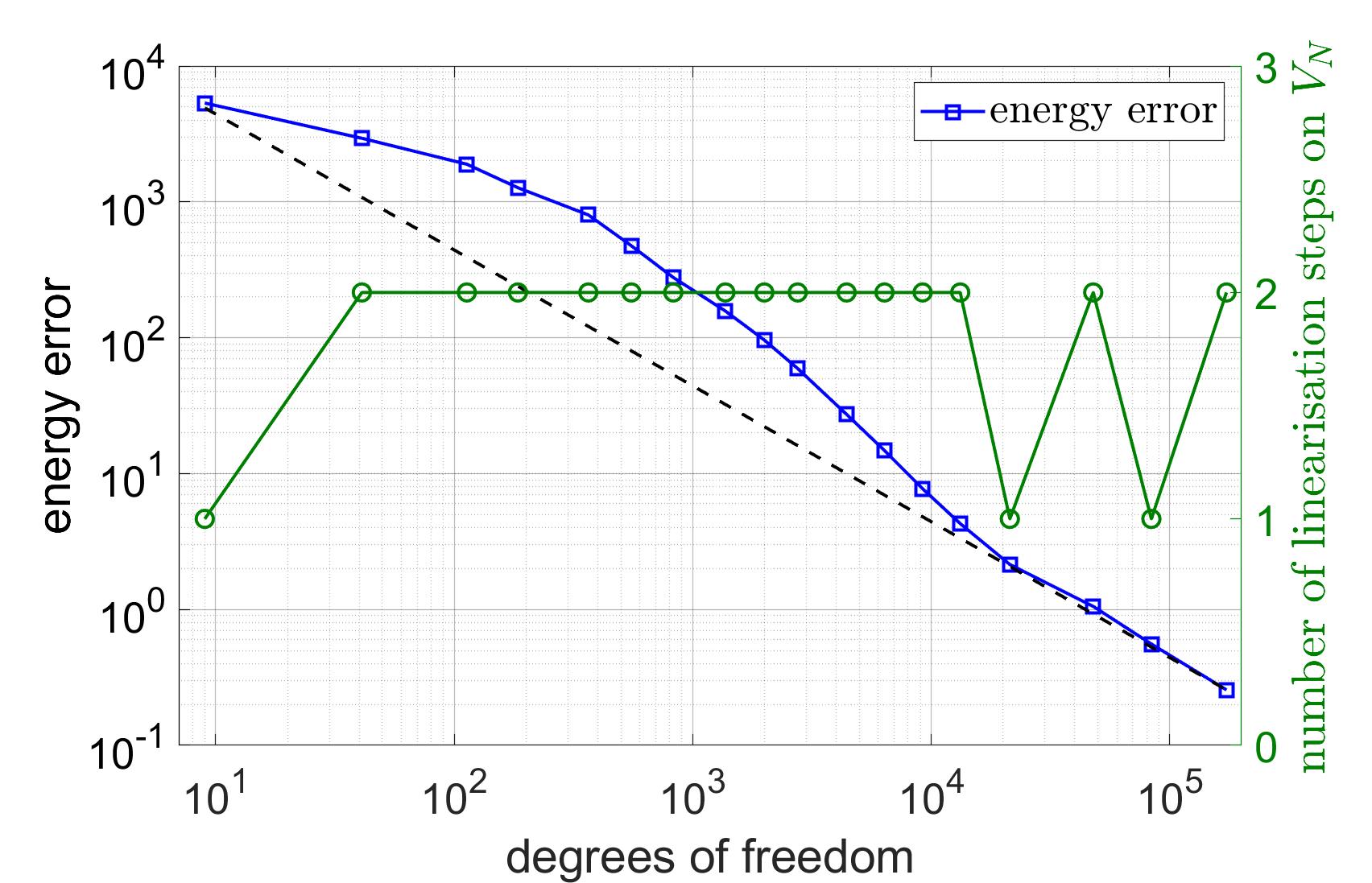}
\hfill
	\includegraphics[width=0.49\textwidth]{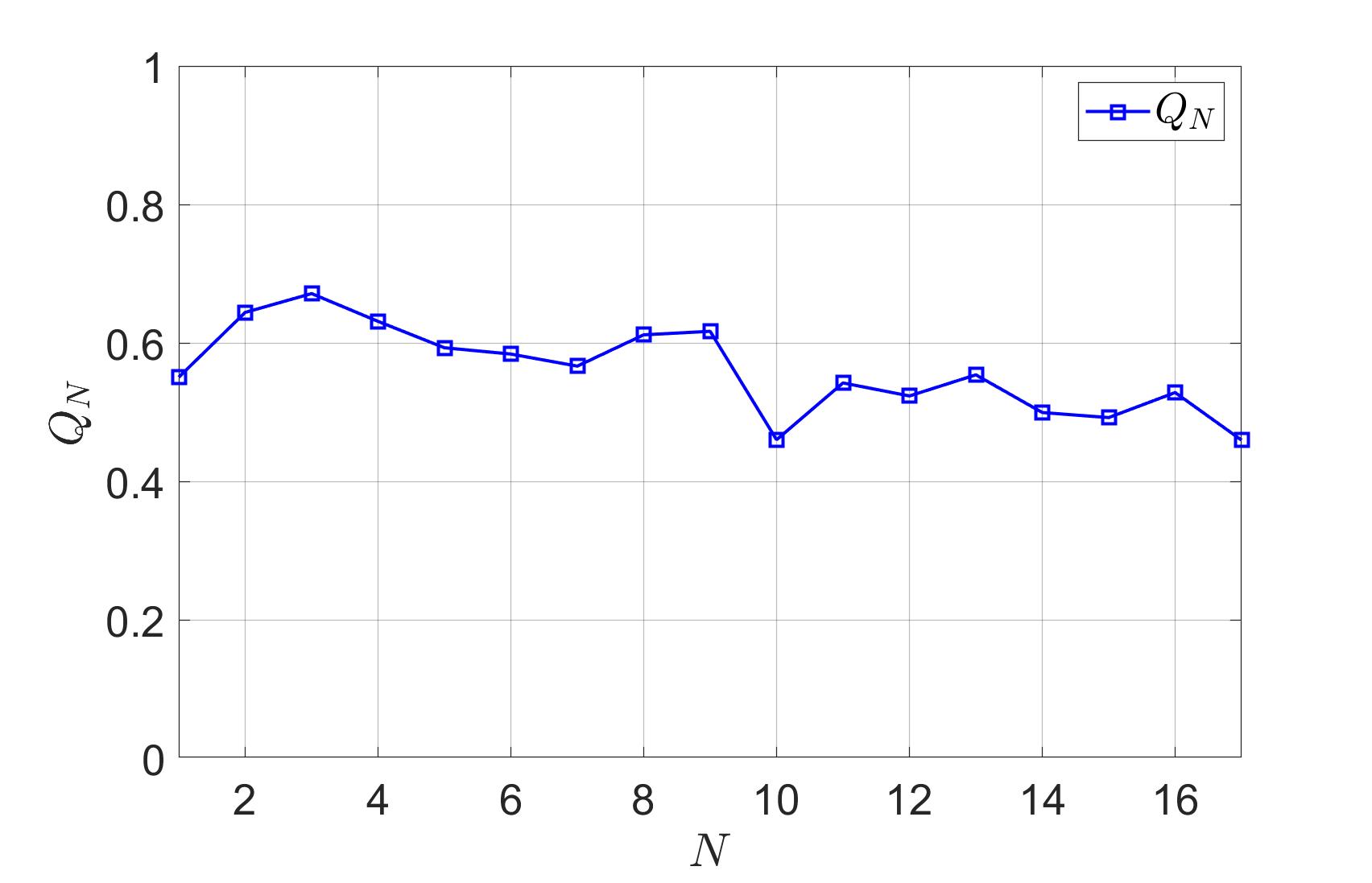}
	\hfill
	\caption{Experiment~\ref{exp:SineGordonEnergy}. Left: Convergence plot for the energy error. Right: Plot of the energy contraction factor $Q_N$.}
	\label{fig:SineGordonQuotient}
\end{figure}

\experiment \label{exp:Lshape}
Let us examine the problem 
\begin{equation*}
\begin{aligned}
-\varepsilon \Delta u &= \exp(-u^2)  \quad&& \text{in } \Omega, \\
u&=0 \quad && \text{on } \partial \Omega,  
\end{aligned}
\end{equation*}
on the L-shaped domain $\Omega:=(0,2)^2 \setminus [1,2] \times [0,1]$, with $\varepsilon=10^{-2}$. For the source function $f(u):=\varepsilon^{-1}\exp(-u^2)$ we find that
\[
\sigma_f(\lambda)=\frac{\sqrt2}{\varepsilon}\exp\left(-\nicefrac12\right)+\frac{1}{\lambda} 
<\frac{1}{\varepsilon}+\frac{1}{\lambda},
\] 
which shows, for any fixed $\rho \geq \nicefrac{1}{\varepsilon}$, that $\lambda \in \Lambda_f(\rho)$ for all $\lambda>0$. Recalling Proposition~\ref{pr:res}, we find that the step size $\Delta t=\varepsilon$ is an admissible choice for our numerical experiment. As before, the convergence rate of the energy error is optimal, see Figure~\ref{fig:Lshape} (right). Furthermore, the space energy contraction factor is clearly bounded away from one, and is even slightly decreasing for an increasing number of refinement steps, cf.~Figure~\ref{fig:LshapeQuotient}.

\begin{figure}
	\hfill
	\includegraphics[width=0.49\textwidth]{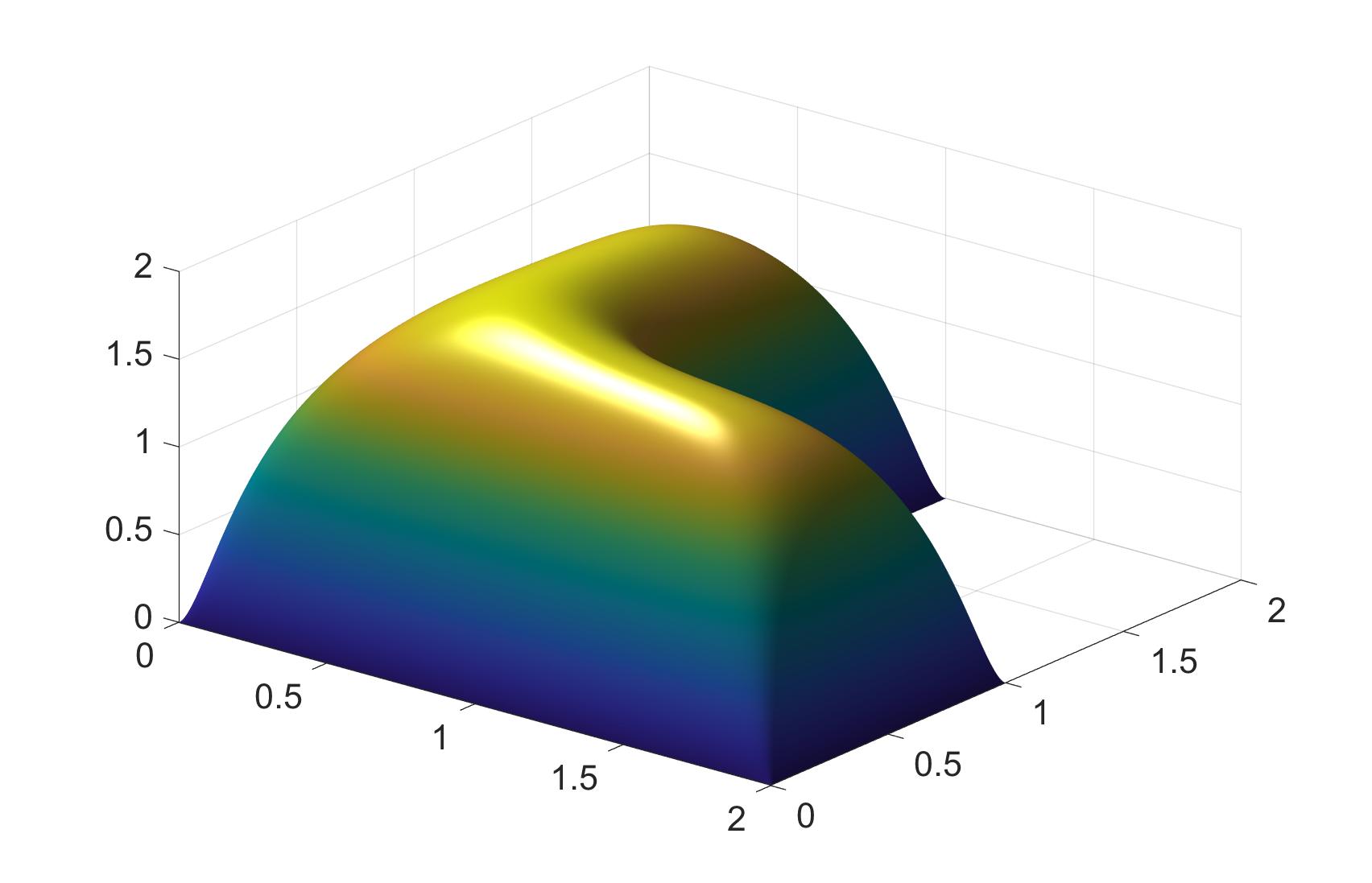}
	\hfill
	\includegraphics[width=0.49\textwidth]{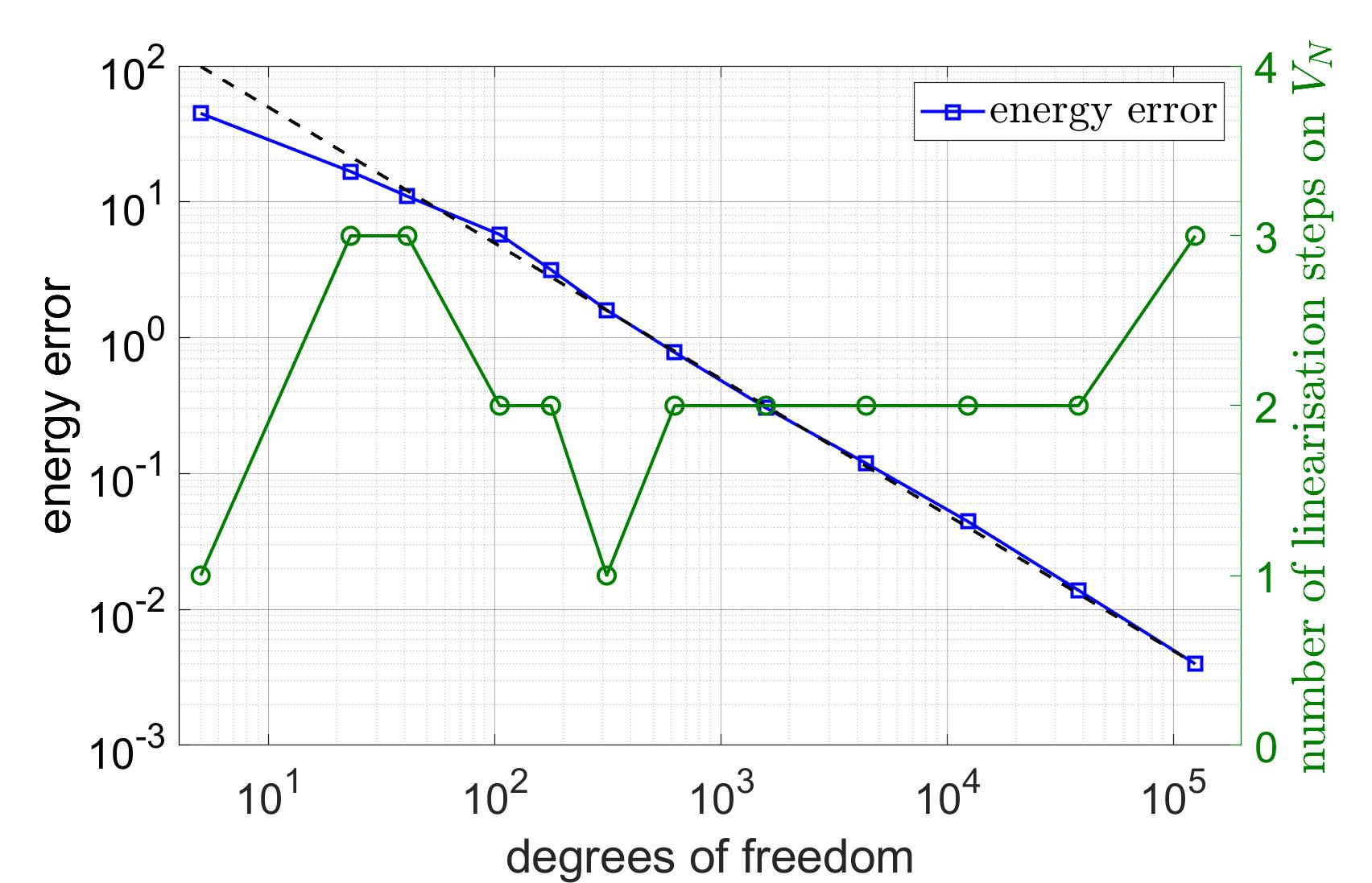}
	\hfill
	\caption{Experiment~\ref{exp:Lshape}. Left: Approximated solution. Right: Convergence plot for the energy error.}
	\label{fig:Lshape}
\end{figure}

\begin{figure}
\centering
	\includegraphics[width=0.49\textwidth]{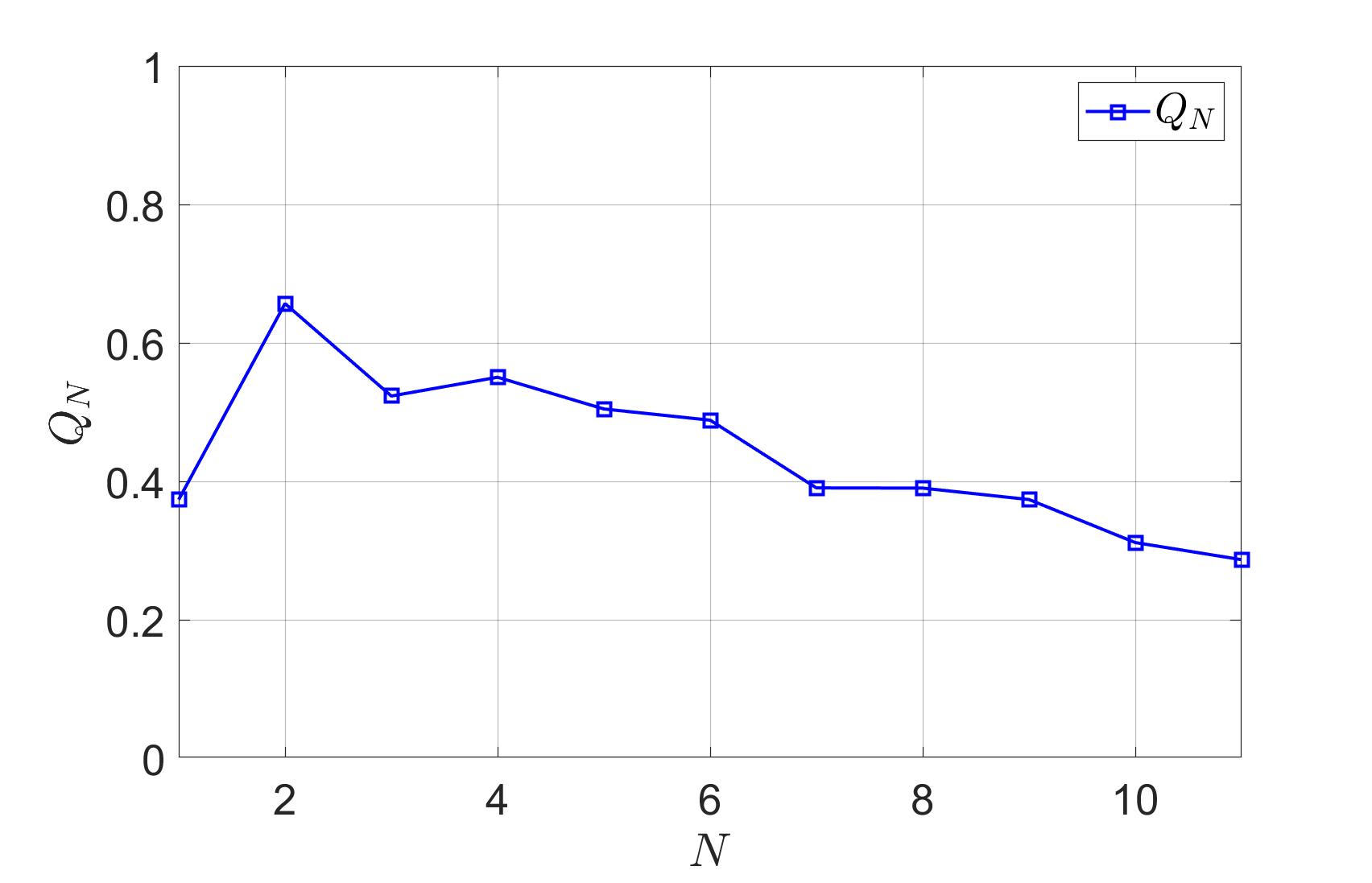}
	\caption{Experiment~\ref{exp:Lshape}: Plot of the energy contraction factor $Q_N$.}
	\label{fig:LshapeQuotient}
\end{figure}

\experiment \label{exp:Arrhenius}
Next we will consider a diffusion-reaction equation with an Arrhenius type production term on the unit square $\Omega:=(0,1)^2$. More specifically, we seek a weak solution $u \in \H^1(\Omega)$ of the Dirichlet boundary value problem
\begin{equation*}
\begin{aligned}
-\Delta u &= (1-|u|)\exp(-\nicefrac{1}{|u|})  \quad&& \text{in } \Omega, \\
u&=2 \quad && \text{on } \partial \Omega;  
\end{aligned}
\end{equation*}
evidently, by standard manipulations, this can be transformed into a problem with homogeneous Dirichlet boundary conditions. We select the step size $\Delta t=1$. Here, since the underlying energy functional involves the exponential integral function $\Ei(\cdot)$, whose evaluation consumes considerable computational time, we determine the reference energy value by running our algorithm until the number of degrees of freedom exceeds $10^5$, and subsequently, for our plots, stop the algorithm as soon as the number of degrees of freedom is greater than $10^4$. Once more, an (almost) optimal convergence rate can be observed in Figure~\ref{fig:Arrhenius} (left). Furthermore the quotient $Q_N$ from~\eqref{eq:QN} is uniformly bounded from above by $0.6$, see Figure~\ref{fig:Arrhenius} (right).

\begin{figure}
\hfill
	\includegraphics[width=0.49\textwidth]{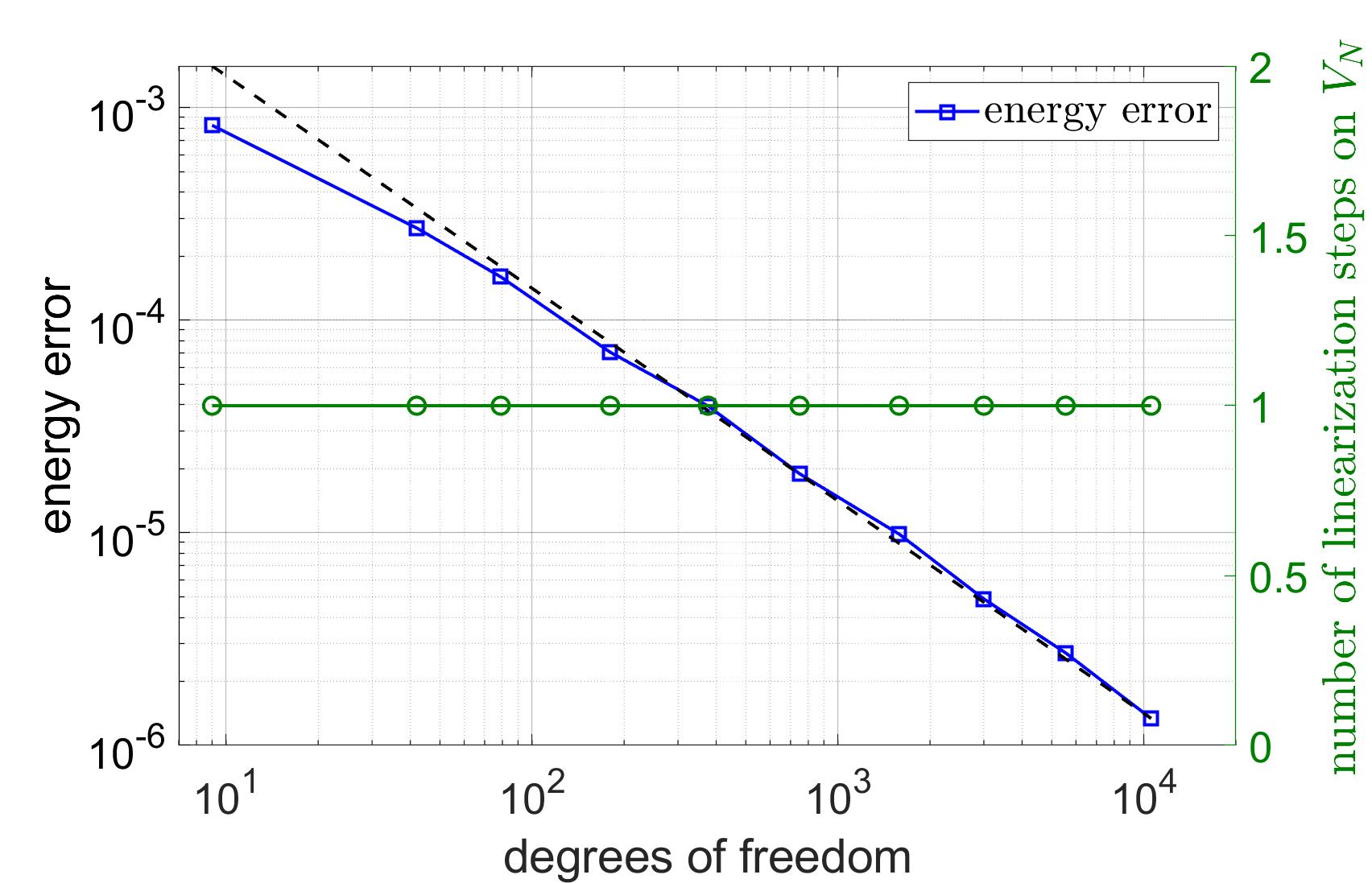}
	\hfill
	\includegraphics[width=0.49\textwidth]{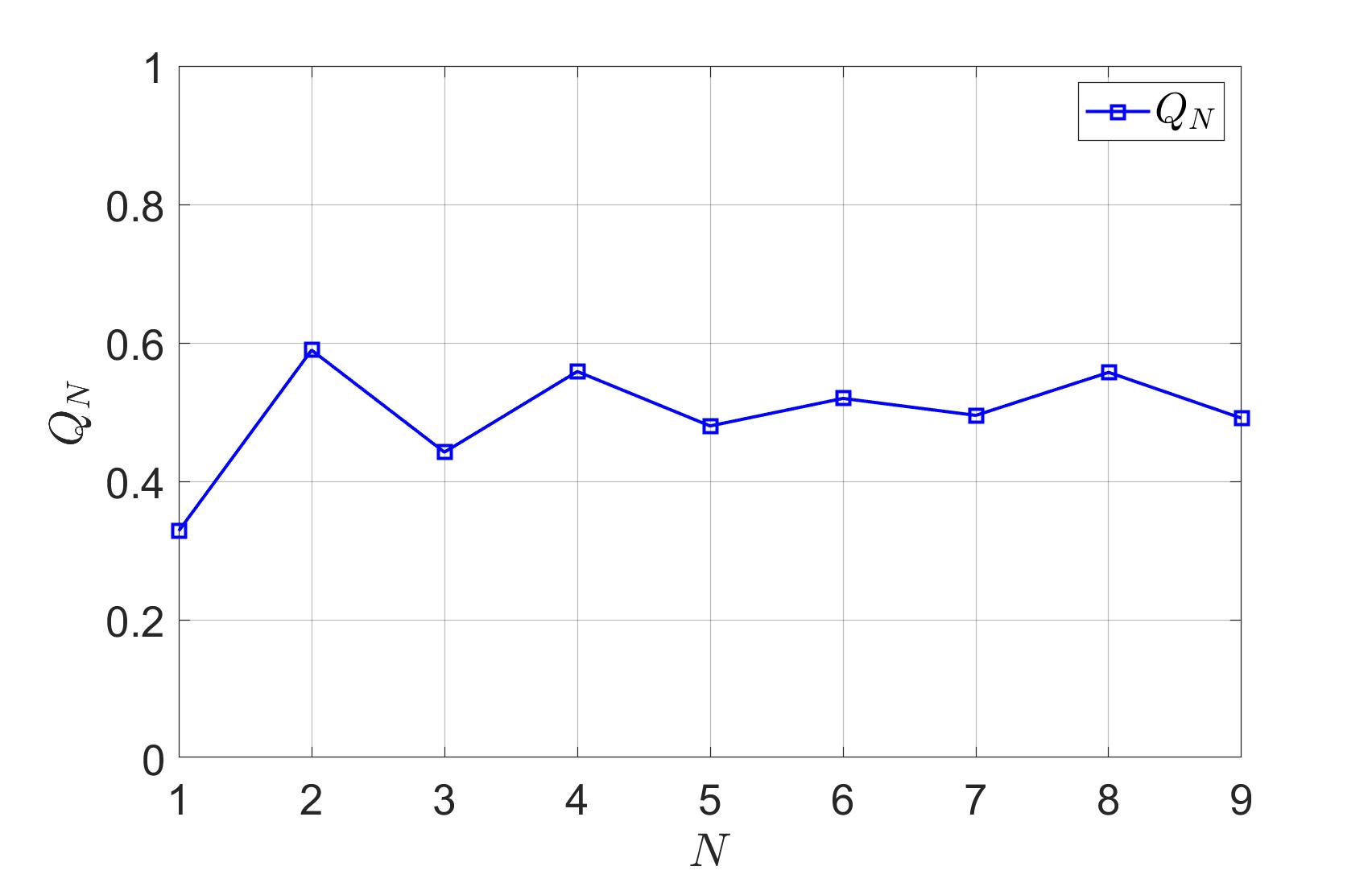}
	\hfill
	\caption{Experiment~\ref{exp:Arrhenius}. Left: Convergence plot for the energy error. Right: Plot of the energy contraction factor $Q_N$.}
	\label{fig:Arrhenius}
\end{figure}

\experiment \label{exp:sign}
For a further demonstration of the effectiveness of our proposed mesh refinement strategy, we study a singularly perturbed problem, with a reaction term that features a discontinuity in the domain $\Omega:=(0,1)^2$: 
\begin{equation*}
\begin{aligned}
-10^{-8} \Delta u &= -u + \mathrm{sign}(x-\nicefrac{1}{2})  \quad&& \text{in } \Omega, \\
u&=0 \quad && \text{on } \partial \Omega.
\end{aligned}
\end{equation*} 
Since this problem is linear, the iteration procedure~\eqref{eq:iterationW} consists of one single step (and solves~\eqref{eq:weakW} exactly). In addition to the boundary layers near $\partial \Omega$, there is also a layer along the discontinuity $x=\nicefrac{1}{2}$ as depicted in Figure~\ref{fig:SignMesh} (left). Still we obtain once more a convergence rate that is close to optimal, see Figure~\ref{fig:sign} (left), thanks to the local mesh refinements along the layers in the solution as illustrated in Figure~\ref{fig:SignMesh} (right). Finally, the contraction quotient $Q_N$ is bounded from above by $0.72$, at least for the region portrayed in Figure~\ref{fig:sign} (right). 

\begin{figure}
	\hfill
	\includegraphics[width=0.49\textwidth]{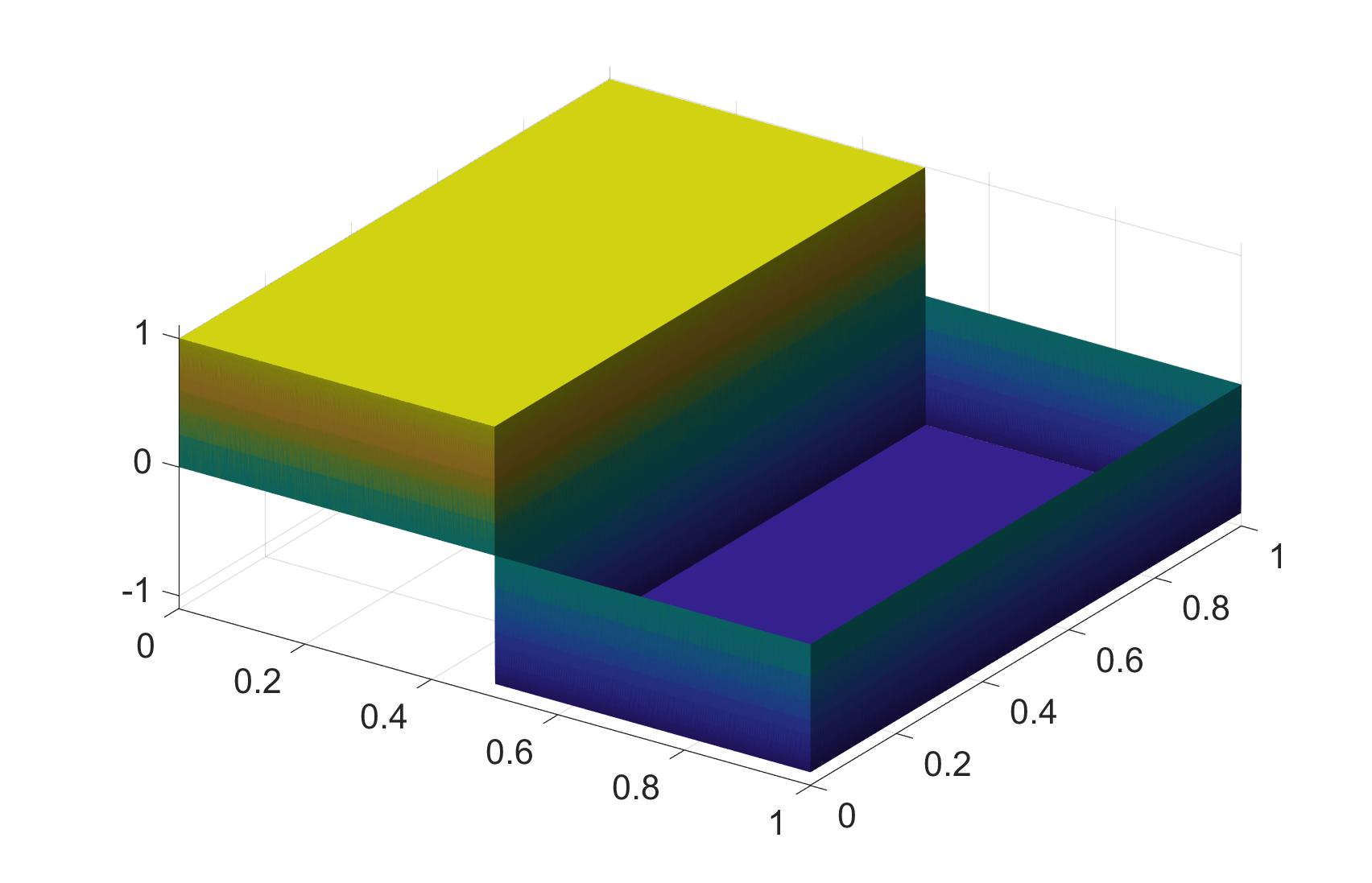}
	\hfill
	\includegraphics[width=0.49\textwidth]{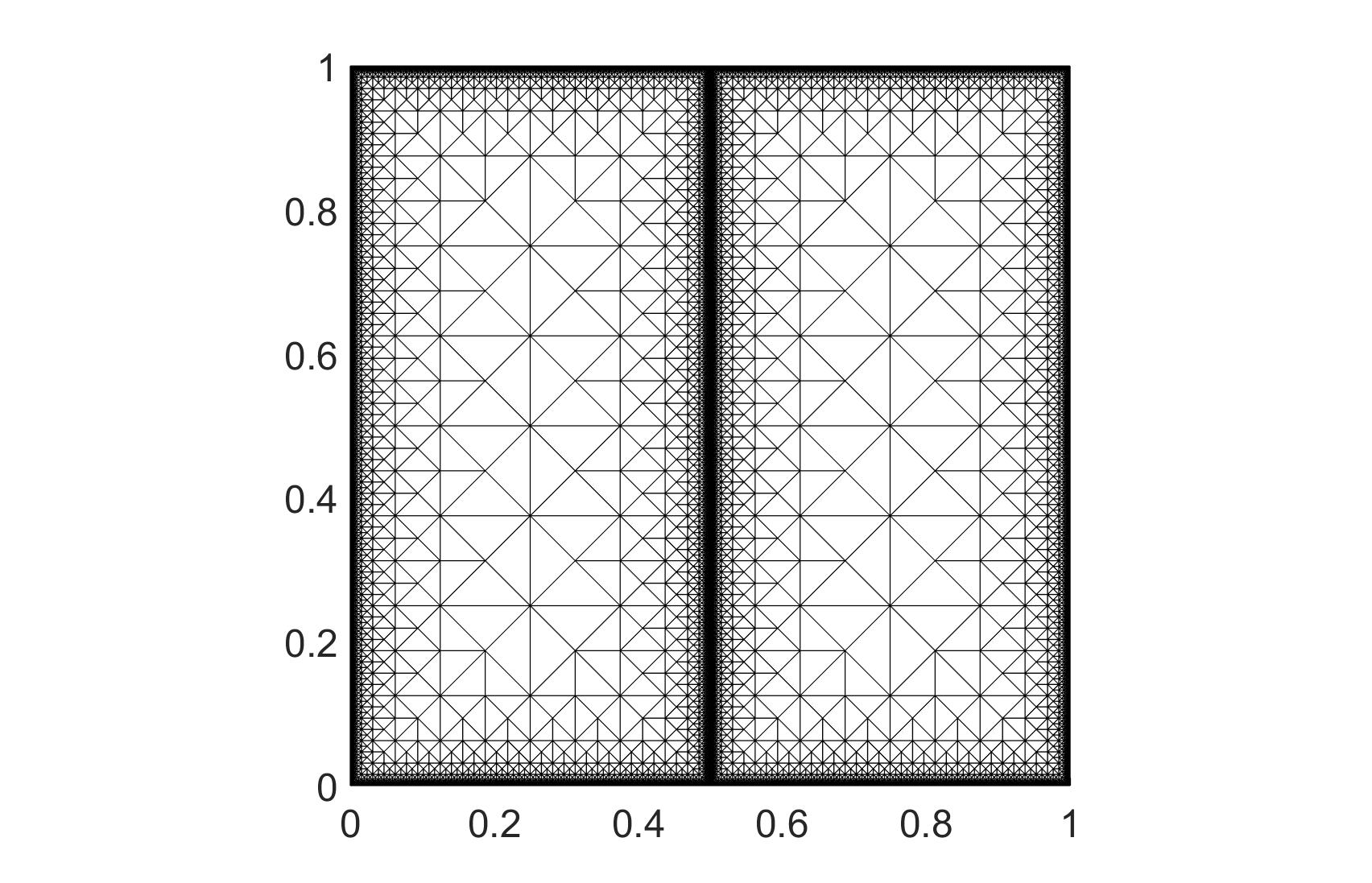}
	\hfill
	\caption{Experiment~\ref{exp:sign}. Left: Approximated solution. Right: Adaptively refined mesh.}
	\label{fig:SignMesh}
\end{figure}

\begin{figure}
\hfill
	\includegraphics[width=0.49\textwidth]{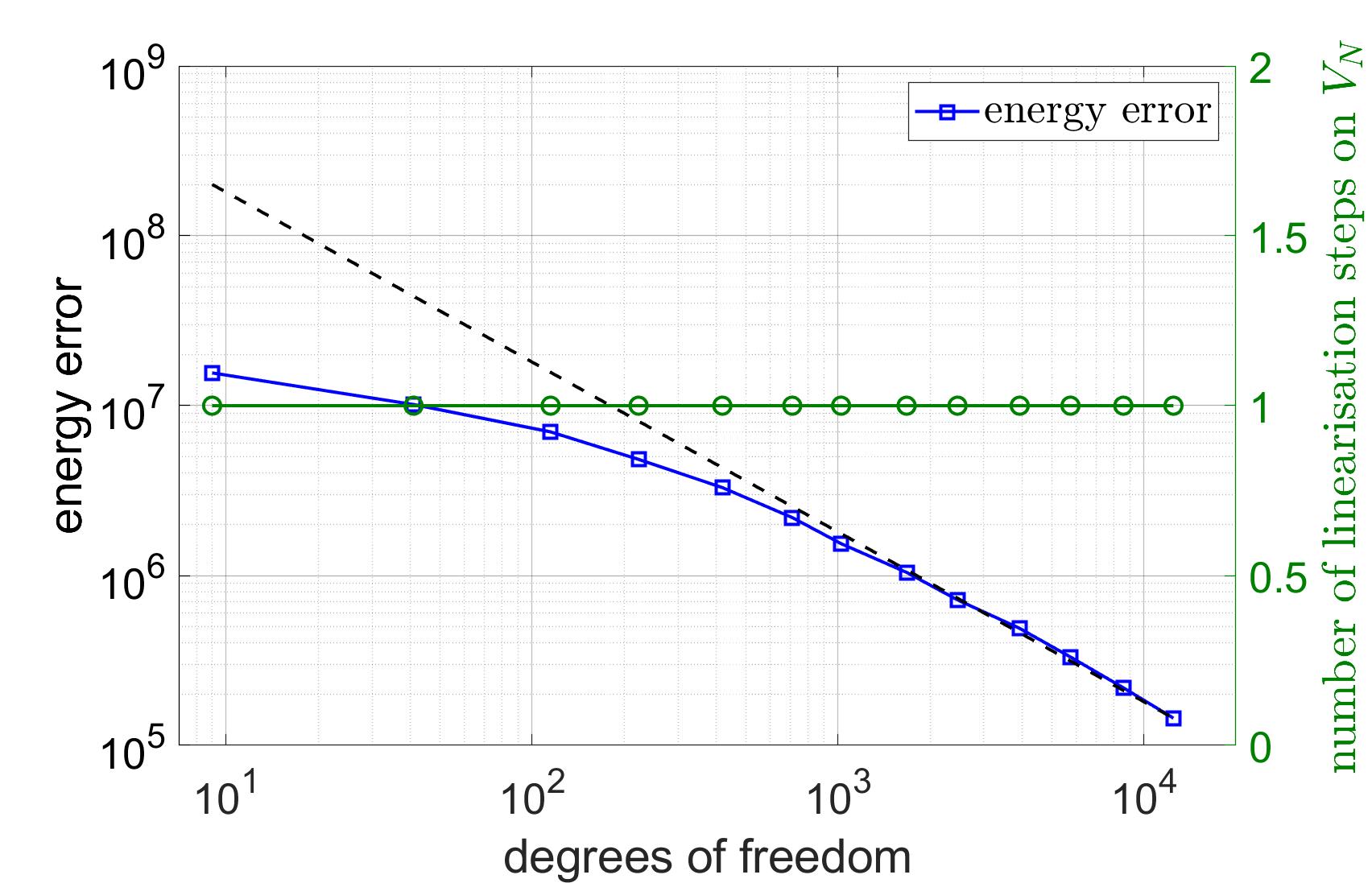}
	\hfill
		\includegraphics[width=0.49\textwidth]{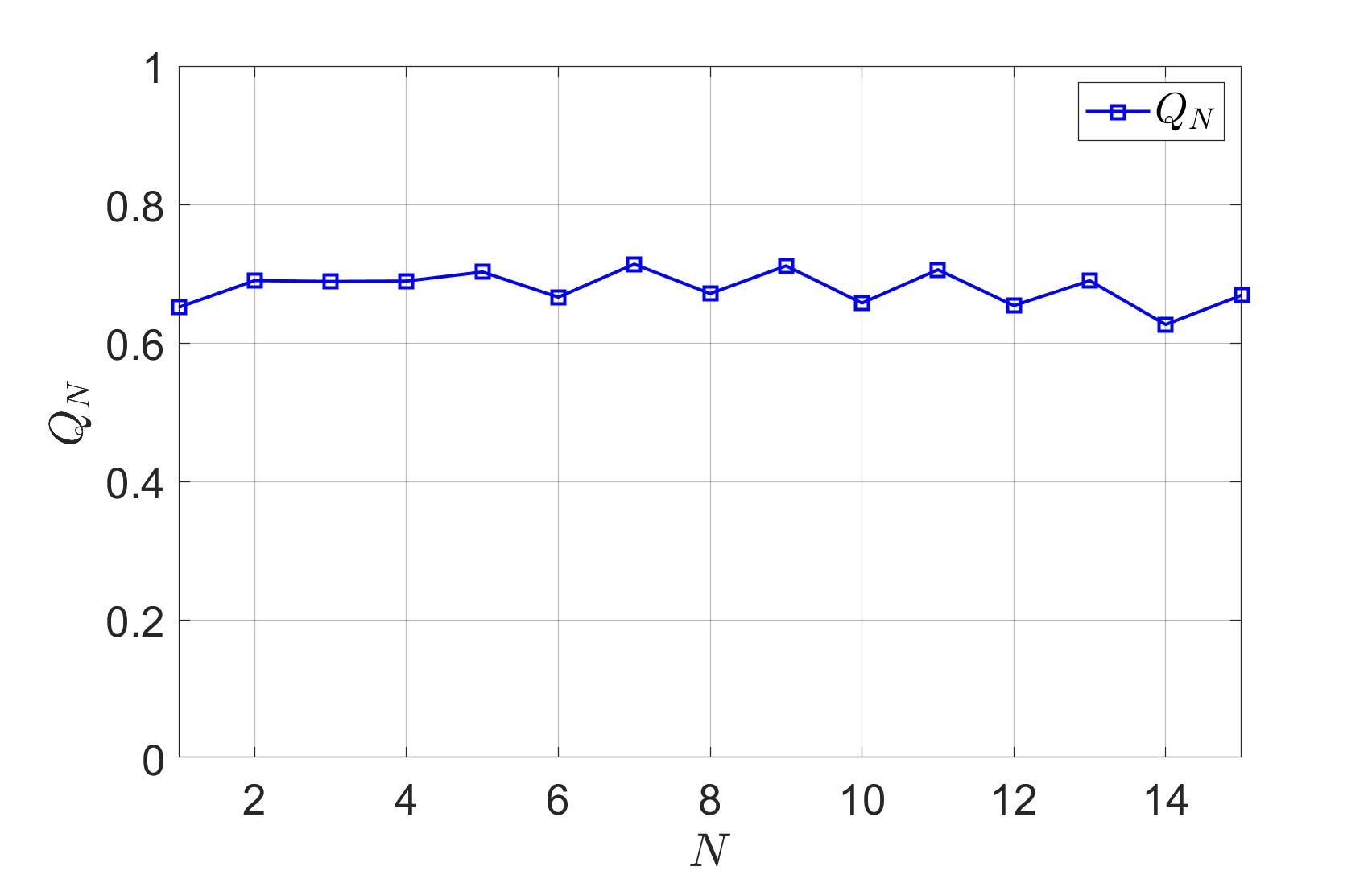}
	\hfill
	\caption{Experiment~\ref{exp:sign}. Left: Convergence plot for the energy error. Right: Plot of the energy contraction factor $Q_N$.}
	\label{fig:sign}
\end{figure}

\experiment \label{exp:oscillation}
In our final experiment, we study a linear singularly perturbed problem with a sign-changing reaction coefficient in the underlying domain $\Omega:=(0,1)^2$:
\begin{equation*}
\begin{aligned}
-\varepsilon \Delta u(x,y) &=(x-\nicefrac12)u(x,y) + 1  \quad&& \text{in} \ \Omega, \\
u(x,y)&=0 \quad && \text{on} \ \partial \Omega.  
\end{aligned}
\end{equation*} 
For $0<\varepsilon\ll 1$, we expect oscillations in the solution as the partial differential operator is no longer coercive in areas where the factor $\varepsilon^{-1}(x-\nicefrac12)$ is strongly positive, cp.~\cite[Example~2]{MelenkWihler:15}. In order to account for this deficiency, we start our computations with a sufficiently fine initial mesh. Furthermore, since the problem is again linear, as in the previous Example~\ref{exp:sign}, we directly solve the discrete system, i.e., without employing the iteration scheme~\eqref{eq:iterationW}. In Figure~\ref{fig:oscillationsurf} (left), for $\varepsilon=10^{-3}$, we observe that the approximated solution indeed exhibits oscillations in the region $\Omega_{+}:=\{(x,y) \in \Omega: x-0.5>0\}$. The energy tends to an optimal rate of $\mathcal{O}(\dof^{-1})$, see Figure~\ref{fig:oscillationsurf} (right), and the quotient $Q_N$ is approximately $0.5$, cf.~Figure~\ref{fig:oscillationquotient}.

\begin{figure}
	\hfill
	\includegraphics[width=0.49\textwidth]{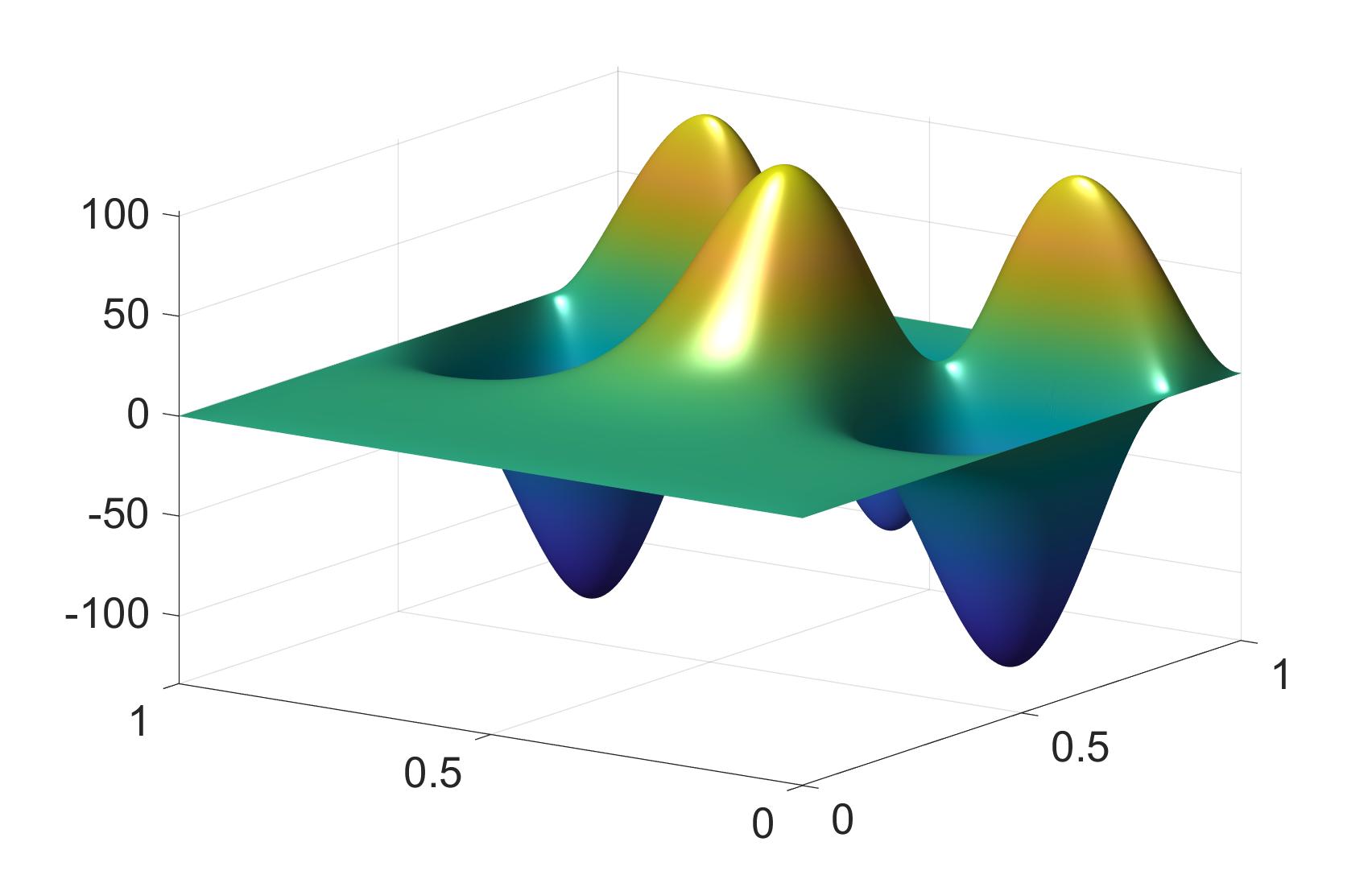}
	\hfill
	\includegraphics[width=0.49\textwidth]{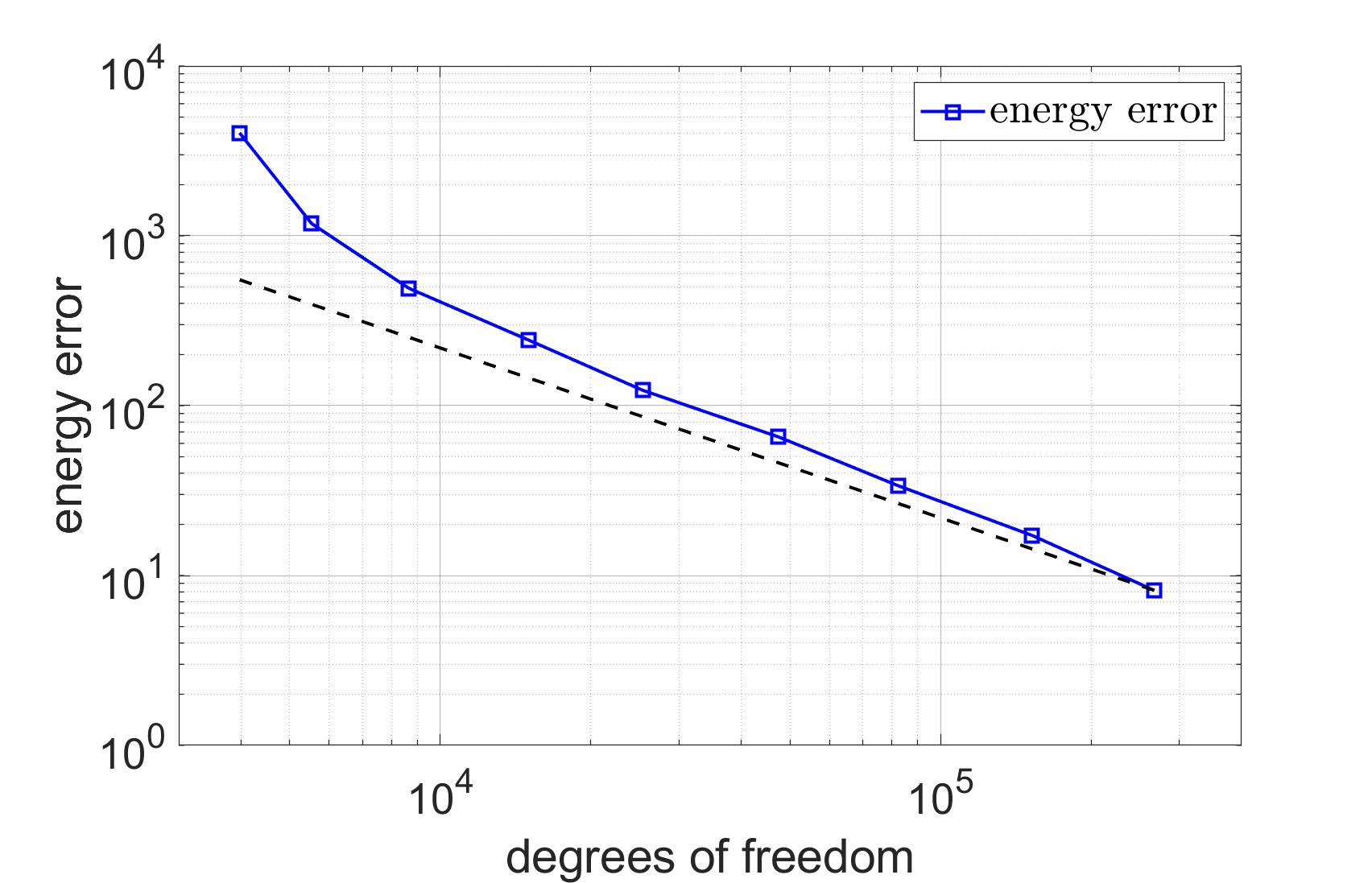}
	\hfill
	\caption{Experiment~\ref{exp:oscillation}. Left: Approximated solution. Right: Convergence plot for the energy error.}
	\label{fig:oscillationsurf}
\end{figure}

\begin{figure}
\centering
	\includegraphics[width=0.49\textwidth]{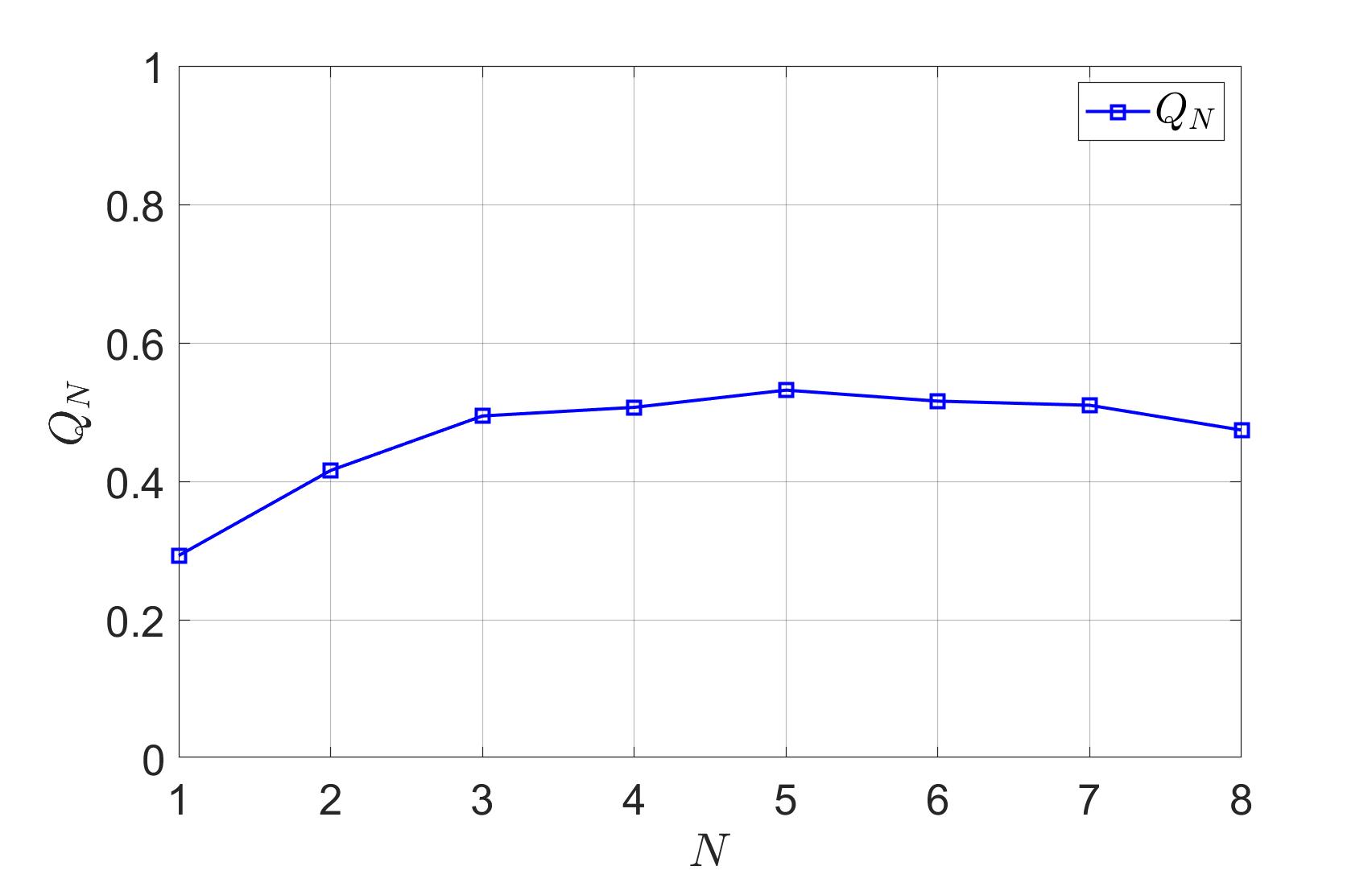}
	\caption{Experiment~\ref{exp:oscillation}: Plot of the energy contraction factor $Q_N$.}
	\label{fig:oscillationquotient}
\end{figure}

\section{Conclusions}
\label{sec:concl}
In this work we have introduced a new analysis and computational procedure for  semilinear diffusion-reaction boundary value problems. The focus of our approach is on a variational framework that induces two very natural components: an energy-minimisation iterative linearisation scheme and a local energy-driven adaptive mesh refinement strategy. These two practical tools are combined in an intertwined manner for the purpose of building an effective numerical solution algorithm. Indeed, our computational experiments clearly demonstrate, even for challenging singularly perturbed examples, that nonlinear reaction terms as well as possible singular layers in the underlying solutions are simultaneously resolved at an empirically optimal rate. The theoretical analysis of the energy-decreasing adaptive finite element space enrichment procedure, in particular, the derivation of sufficient conditions for the bound~\eqref{eq:disq} to hold, is subject to future research; this clearly constitutes, however, a topic on its own and ranges beyond the scope of the present paper.

\bibliographystyle{amsplain}
\bibliography{references}
\end{document}